\documentclass[leqno,journal=PRIMS,lang=british]{ems-journal-prims}
\usepackage[all]{xy}
\usepackage{amssymb,amsmath}
\usepackage{enumerate}

\usepackage{amsfonts}

\usepackage{eucal, amsfonts, latexsym,textcomp}

\usepackage{amsmath}

\usepackage{bbm}
\usepackage{amscd,amsfonts}

\usepackage{pifont}

\usepackage{color}
\usepackage{mathrsfs}

\usepackage{indentfirst,bm,fancyhdr,dsfont}

\usepackage{graphicx}
\usepackage[all]{xy}

\usepackage{tikz}

\DeclareMathOperator{\im}{im}
\usepackage{bbm}

\usepackage{indentfirst,fancyhdr,dsfont}

\usepackage{mathrsfs,color}

\newtheorem{thm}{Theorem}[section]
\newtheorem{cor}[thm]{Corollary}
\newtheorem{lem}[thm]{Lemma}

\newtheorem{prop}[thm]{Proposition}

\newtheorem{question}[thm]{Question}

\theoremstyle{definition}
\newtheorem{defn}[thm]{Definition}
\newtheorem{rem}[thm]{Remark}

\newtheorem{rems}[thm]{Remarks}

\numberwithin{equation}{thm}
\newtheorem{conven}[thm]{Convention}

\numberwithin{equation}{section}

\def\ggg{\mathfrak{g}}
\def\sss{\mathfrak{s}}
\def\lll{\mathfrak{l}}
\def\zzz{\mathfrak{z}}
\def\ppp{\mathfrak{p}}
\def\nnn{\mathfrak{n}}
\def\hhh{\mathfrak{h}}
\def\gl{\mathfrak{g}}
\def\ker{{\rm Ker\,}}
\def\im{{\rm Im\,}}

\def\Hom{{\rm Hom}}
\def\Ext{{\rm Ext}}

\def\bba{\mathbb A}
\def\bbf{\mathbb F}
\def\bbz{\mathbb Z}

\def\co{\mathcal O}
\def\calr{\mathcal R}
\def\cv{\mathcal{V}}
\def\calv{\mathcal{V}}
\def\cala{\mathcal{A}}

\def\sfp{\textsf{P}}
\def\sfd{\textsf{d}}
\def\bfC{\textbf{C}}
\def\bfV{\textbf{V}}
\def\bfH{\textbf{H}}
\def\scrl{\mathscr{L}}

	\def\lc{\mathscr{L\kern -.4em   C}}
\def\ulc{{{}^u{\kern -.3em}\mathscr{L\kern -.4em C}}}
\def\uE	{{{}^u{\kern -.1em}\mathcal{E}}}

\def\hmod{\text{-}\textbf{mod}}

\communicated{T. Arakawa}
\received{6 February 2024}
\revised{6 March 2025}

\begin{document}



\title{The Category $\mathcal{O}$  for Lie algebras of vector fields (II): Lie-Cartan modules and cohomology}

\TitleHead{Lie-Cartan modules and cohomology}



\emsauthor*{1}{
	\givenname{Feifei}
	\surname{Duan}
	\mrid{1234567}
	\orcid{0000-0001-0002-0003}}{Feifei~Duan}
\emsauthor{2}{
	\givenname{Bin}
	\surname{Shu}
	\mrid{}
	\orcid{}}{Bin Shu}
\emsauthor{3}{
	\givenname{Yufeng}
	\surname{Yao}
	\mrid{}
	\orcid{}}{Yufeng~Yao}
\emsauthor{4}{
	\givenname{Priyanshu}
	\surname{Chakraborty}
	\mrid{}
	\orcid{}}{Priyanshu~Chakraborty}

\Emsaffil{1}{
	\pretext{}
	\department{School of Mathematical Sciences \& Hebei Key Laboratory of Computational Mathematics and Applications}
	\organisation{Hebei Normal University}
	\rorid{01a2bcd34}
	\address{No. 20 Nanhuan East Rd}
	\zip{050024}
	\city{Shijiazhuang}
	\country{China}
	\posttext{}
	\affemail{duanfeifei0918@126.com}
	\furtheremail{}}
\Emsaffil{2}{
	\pretext{}
	\department{1}{School of Mathematical Sciences,Ministry of Education Key Laboratory of Mathematics and Engineering Applications \& Shanghai Key Laboratory of PMMP}
	\organisation{1}{East China Normal University}
	\rorid{1}{}
	\address{1}{No. 500 Dongchuan Rd.}
	\zip{1}{200241}
	\city{1}{Shanghai}
	\country{1}{China}
	\posttext{}
	\pretext{}
	\department{2}{} 
	\organisation{2}{}%
	\rorid{2}{}
	\address{2}{}%
	\zip{2}{}
	\city{2}{}
	\country{2}{} 
	\posttext{}
	\affemail{bshu@math.ecnu.edu.cn}
	\furtheremail{}}
\Emsaffil{3}{
	\pretext{}
	\department{Department of Mathematics}
	\organisation{Shanghai Maritime University}
	\rorid{}
	\address{No. 1550 Haigang Ave.}
	\zip{201306}
	\city{Shanghai}
	\country{China}
	\posttext{}
	\affemail{yfyao@shmtu.edu.cn}
	\furtheremail{}}

\Emsaffil{4}{
	\pretext{}
	\department{School of Mathematical Sciences, Ministry of Education Key Laboratory of Mathematics and Engineering Applications \& Shanghai Key Laboratory of PMMP}
	\organisation{East China Normal University}
	\rorid{}
	\address{No. 500 Dongchuan Rd.}
	\zip{200241}
	\city{Shanghai}
	\country{China}
	\posttext{}
	\affemail{priyanshu@math.ecnu.edu.cn}
	\furtheremail{}}
	


\maketitle

\begin{abstract}
As a sequel to \cite{DSY},  we introduce here a category \texorpdfstring{$\lc$ }{} arising from the BGG category \texorpdfstring{$\co$ }{} defined in \cite{DSY} for Lie algebras of polynomial vector fields.  The objects of \texorpdfstring{$\lc$ }{}  are so-called Lie-Cartan modules which admit both Lie-module structure and compatible \texorpdfstring{$\calr$}{}-module structure (\texorpdfstring{$\calr$}{} denotes the corresponding polynomial ring). This terminology is natural, coming from affine connections in differential geometry through which the structure sheaves in topology and the vector fields in geometry are integrated for differential manifolds.
			
In this paper, we study Lie-Cartan modules and their categorical  and cohomology properties. 
The category $\lc$ is abelian, and  a ``highest weight category" with depths. Notably, the set of co-standard objects in the category $\mathcal{O}$
turns out to represent the isomorphism classes of simple objects of
\texorpdfstring{$\lc$.} {}
We then establish the cohomology for the category of universal Lie-Cartan modules (called the $\ulc$-cohomology), extending Chevalley-Eilenberg cohomology theory. Another notable  result says that in the fundamental case $\ggg= W(n)$,  the extension ring $\Ext^\bullet_\ulc(\calr,\calr)$ for the polynomial algebra $\calr$ in the $\ulc$-cohomology is isomorphic to the usual cohomology ring  $H^\bullet(\frak{gl}(n))$ of the general linear Lie algebra $\mathfrak{gl}(n)$.
\end{abstract}


\section{Introduction}

\subsection{Background and Motivation}	This paper is one of series of work on representations for Lie algebras of vector fields $\ggg\in \{W(n)$, $S(n)$, $H(n)\}$.  In the previous one \cite{DSY}, the first three authors  introduced a parabolic BGG category and studied indecomposable tilting modules and their character formulas for $\ggg$. The purpose of the present paper is twofold. One is to introduce Lie-Cartan modules and classify all irreducible Lie-Cartan modules. The other one is to develop the cohomology theory for (universal) Lie-Cartan modules and make its description and some important computations.

 Let us first recall some background. Associated with an affine space $\textsf{E}$, the Lie algebras of vector fields on $\textsf{E}$ are basic algebraic objects. When considering the fundamental case $\textsf{E}=\bba^n$, we have the Lie algebras of vector fields $W(n)$, $S(n)$, $H(n)$ and $K(n)$. Those Lie algebras are  involved in the classification of transitive Lie pseudogroup raised by E. Cartan (cf. \cite{GS}, \cite{Ko}, \cite{KoN}, \cite{SS}, etc.), and also involved in the classification of finite dimensional simple Lie algebras over an algebraically closed filed of prime characteristic (cf. \cite{Kac} and \cite{KS}, etc.).
		
Recall that the infinite dimensional Lie algebra $\ggg=X(n)$, $X\in\{W,S,H\}$, is endowed with a canonical graded structure
		$$\ggg=\sum_{i=-1}^\infty \ggg_{i}$$
		arising from the grading of polynomials from  $\calr:=\bbf[x_1,\cdots,x_n]$, the coordinate ring of $\bba^n$. As homogeneous spaces, $\ggg_{{-1}}$ is spanned by all partial derivations $\partial_i$, $i=1,\cdots,n$, and $\ggg_{0}$ is isomorphic to $\mathfrak{gl}(n)$, $\mathfrak{sl}(n)$ or $\mathfrak{sp}(n)$, containing a canonical maximal torus $\hhh$. We consider the subalgebra $\sfp=\ggg_{{-1}}\oplus \ggg_{0}$ which can be regarded as a parabolic subalgebra. Associated with $\sfp$, we introduce a subcategory $\co$ of $\ggg$-module category, an analogue of the parabolic BGG category over complex semi-simple Lie algebras, whose objects satisfy the axioms (see Definition \ref{category O}) including local finiteness over $U(\sfp)$. Such an option concerning $\sfp$  enables us to get good understanding on representations for $\ggg$ (see \cite{DSY}). In the same spirit, we also studied the representations of finite-dimensional Lie superalgebras in \cite{DSY2} where we made an exhaustive and essential explanations on the option of $\sfp$.

\subsection{Lie-Cartan modules} 
In the present paper, we first introduce Lie-Cartan modules for $\ggg=W(n), S(n)$ and $H(n)$
(see Definition \ref{def: lc mod}). Generally speaking, for a Lie algebra $\mathcal{G}$ of vector fields on an algebraic vareity $\textsf{E}$,
 we can consider a kind of modules which are endowed with two module structures over $\mathcal{G}$ and over $\calr$ where $\calr$ means the ring of regular functions on $\textsf{E}$, with compatibility of both modules. This thinking becomes a kind of usual ways, for example, such as affine (integrable)  connections in differential geometry, differential operator rings and related fields  (see \cite{Bj},  \cite{Hel}, \cite{HTT}, \cite{Skr},   {\sl{etc}.}).  In particular,  S. Skryabin introduced the notion Lie-Cartan pairs in \cite{Skr2}. Under the framework of Lie-Cartan pairs, he successfully studied  $\textbf{k}$-forms of a finite-dimensional non-classical simple Lie algebra over an algebraical closure $\overline{\textbf{k}}$ of $\textbf{k}$ where $\textbf{k}$ is an algebraically non-closed field of  prime characteristic. In this way, we  go further in our scope, introducing the category of Lie-Cartan modules as a subcategory of $\co$ whose objects are simultaneously $\calr$-modules satisfying the compatibility $[\rho(\xi),\theta(f)]=\theta(\xi(f))$ for $\ggg$-action $\rho$ and $\calr$-action $\theta$ (see Definition \ref{def: lc mod}). It should be mentioned that after finishing the manuscript we were aware of the works \cite{BFN}, \cite{BIN} by Billig-Futorny-Nilsson,  and by  Billig-Ingalls-Nasr  respectively,  where the authors studied $\mathcal{AV}$ modules in the same spirit as above.
 Roughly speaking, Lie-Cartan modules can be  regarded as  $\mathcal{AV}$-modules in the category $\mathcal{O}$ which parallels  to the BGG category of complex semisimple Lie algebras.

 The most important examples of Lie-Cartan modules appearing in \cite{DSY} are the co-induced modules $\text{Coind}^{U(\ggg)}_{U(\ggg_{\geq0})}(L^0(\lambda))$ with $L^0(\lambda)$ running finite-dimensional irreducible representations of $\ggg_0$ (with trivial action of  $\ggg_{>0}$), or to say, with $\lambda$ running over $\Lambda^+,$ the set of dominant integral weights of $\hhh$. These modules are also regarded as co-standard modules with respect to tilting module theory (or with respect to the category theory of  highest weight modules, see \cite{Hum}), which can be realized as $\cv(\lambda)$ by prolongation (see \S\ref{sec: prolong real} for the notation, and see \cite{Rud, Rud2, Shen1, Shen,  Skr} for more explanations). We thoroughly investigate the category of Lie-Cartan modules, and finally accomplish the classification of all simple objects. It shows that the co-standard modules are not only irreducible Lie-Cartan modules, but also present all isomorphism classes of irreducible Lie-Cartan modules up to depths. This is the first main result of the paper. Precisely, we have

\begin{thm} The set $\{{}^{d}\hspace{-2pt}\cv(\lambda)\mid \lambda\in {\Lambda^+}, d\in\bbz\}$ exhausts all non-isomorphic irreducible  Lie-Cartan modules.
\end{thm}
This theorem will be presented in Theorem \ref{thm: irre-mod}. There are two ingredients in the proof. One is the introduction of Lie-Cartan radicals for the finite-depth Lie-Cartan modules (any finitely-generated Lie-Cartan modules admit finite-depth) which plays a key role (see \S\ref{sec: radical}). The other one is a nontrivial result (Extension Lemma) which explains how a $\ggg$-module homomorphism  becomes a Lie-Cartan module homomorphism (see Lemma \ref{lem: ext}).

\subsection{Cohomology} The other topic of this paper is to develop the cohomology of universal Lie-Cartan modules. Here  the prefix ``universal" means that the objects may  not necessarily only come from the category $\co$, instead, may be any $U(\ggg)$-modules with compatible $\calr$-module structure (see Definition \ref{defn: univ LC}). This notion of universal Lie-Cartan modules coincides with that of $\mathcal{AV}$-modules in \cite{BFN}. Our notion of Lie-Cartan modules is raised in the same spirit of  Lie-Cartan pairs introduced in \cite[\S 6]{Skr2}. Actually, this spirit can be applied more (see for example \cite{Wa}).  The  category of universal Lie-Cartan modules  is denoted by $\ulc$. {With aid of introduction of the naturalized algebra $\calr\natural U(\ggg)$ combining $U(\ggg)$ and $\calr$ by the following Lie axiom
$$ [X,f]=X(f)$$
for $X\in\ggg$ and $f\in\calr$ (see Definition \ref{def: nat alg}),  $\ulc$  can be identified with the category of $\calr\natural U(\ggg)$-modules  (see Lemma \ref{lem: tensor r ulc}).}
In the concluding section, we introduce the cohomology of the category  $\ulc$, the so-called $\ulc$-cohomology (see Definition \ref{def: ulc coh}). By definition, the $q$th $\ulc$-cohomology  with coefficient in $M\in\ulc$ is  the right derived functor $\textsf{R}^q(\Gamma)(M),$ where $\Gamma=\Hom_\ulc(\calr,-)$ is a left exact functor from $\ulc$ to  the category of $(\ggg,\calr)$-modules. By the notion of the category of $(\ggg,\calr)$-modules it means that the objects admit both $\ggg$- and $\calr$-module structure, not required to be mutually compatible.

By exploiting the ideas of constructing Chevalley-Eilenberg complex, we construct a complex $\bfC(M)$ for $M\in\ulc$ (see \S\ref{sec: extended CE complex}). Then we make a realization of  the $q$th $\ulc$-cohomology $H^q_\ulc(M)$ (Theorem \ref{thm: cohom thm}). The $\ulc$-cohomology  can be regarded as an extension of Chevalley-Eilenberg cohomology. In this part,  we have the following main result.

\begin{thm}
The following statements hold.
\begin{itemize}
\item[(1)] The cohomology modules $H^q_{\ulc}(M)$ for $M\in \ulc$, $q\in \bbz_{\geq0},$ are the cohomology of the cochain complex $\bfC(M),$
which  means
$$H^q_\ulc(M)=H^q(\bfC(M)).$$

\item[(2)] Let $\ggg=W(n)$. Then the extension ring of $\calr$ in $\ulc$-cohomology satisfies
$$\Ext_\ulc^\bullet(\calr,\calr)\cong H^\bullet(\mathfrak{gl}(n,\bbf))$$
as rings, where $H^\bullet(\mathfrak{gl}(n,\bbf))$  denotes the ordinary cohomology of the general linear Lie algebra.
\end{itemize}
\end{thm}
The above theorem will be presented in Theorems \ref{thm: cohom thm}
and \ref{thm: third main thm} respectively.


\section{Preliminaries}
In this paper, we always assume that the ground field $\mathbb{F}$ is algebraically closed, and of characteristic $0$. All vector spaces (modules) are over $\mathbb{F}$.
		
\subsection{The Lie algebras of vector fields \texorpdfstring{$W(n)$, $S(n)$}{} and \texorpdfstring{$H(n)$}{}}\label{CartanType}
Let $n$ be a positive integer, and $\calr=\mathbb{F}[x_1,\cdots, x_n]$ be the polynomial algebra of $n$ indeterminants. Then $\calr$ admits a natural grading via the degree: $\calr=\sum_{i\geq 0}\calr_i$ with $\calr_0=\bbf$ and $\calr_i$ consists of all homogeneous polynomials of degree $i$ for $i>0$.
		
		Denote by $W(n)$ the Lie algebra of all  derivations  on $\calr$. Then $W(n)$ is a free $\calr$-module with basis $\{\partial_i\mid 1\leq i\leq n\}$, where $\partial_i$ is the partial derivation with respect to $x_i$, i.e., $\partial_i(x_j)=\delta_{ij}$ for $1\leq i,j\leq n$. The natural $\mathbb{Z}$-grading on $\calr$ induces the corresponding $\mathbb{Z}$-grading on $W(n)$, i.e., $W(n)=\bigoplus\limits_{i=-1}^{\infty}W(n)_{i}$, where $W(n)_{i}=\text{span}_{\mathbb{F}}\{f_j\partial_j\mid f_j\in \calr, \text{deg}(f_j)=i+1,1\leq j\leq n\}.$
		
		The Lie algebra $S(n)$ of special type is a subalgebra of $W(n)$ consisting of vector fields $\sum_i f_i\partial_i$ with zero divergence, i.e., $S(n)=\{\sum_i f_i\partial_i\in W(n)\mid \sum_i\partial_i(f_i)=0\}$. By definition, it is easily seen that $S(n)$ is spanned by those elements $D_{ij}(x^{\alpha})$ with $\alpha=(\alpha(1),\cdots, \alpha(n))\in\mathbb{N}^n$,
		$x^{\alpha}=x_1^{\alpha(1)}\cdots x_n^{\alpha(n)}$, and $1\leq i<j\leq n$, where $D_{ij}: \calr\longrightarrow \calr$ is the linear mapping defined by $D_{ij}(x^{\alpha})=
		\alpha_jx^{\alpha-\epsilon_j}
		\partial_i-\alpha_ix^{\alpha-\epsilon_i}\partial_j,\,\forall\,\alpha\in\mathbb{N}^n$,  with $\epsilon_k:=(\delta_{1k},\cdots,\delta_{nk})$ for $k=1,\cdots,n$,  and $\delta_{st}=1$ if $s=t$, or $0$ otherwise. Since the divergence operator is a homogeneous operator of degree $-2$, the algebra $S(n)$ inherits the $\mathbb{Z}$-grading of $W(n)$. {\sl{Hereafter, we abuse the notation $x^\alpha$ for $\alpha\in \bbz^n$, by making the convention that $x^\alpha=0$ unless $\alpha\in \mathbb{N}^n$.}}
		
		When $n=2r$ is even, the elements in $W(n)$ that annihilate the 2-form $\sum_{i=1}^r dx_i\wedge dx_{i+r}$ are called Hamiltonian.  The Lie algebra $H(n)$ of Hamiltonian type is a subalgebra of $W(n)$ consisting of all Hamiltonian elements in $W(n)$. By the definition,  $H(n)$ has a canonical basis $\{D_H(x^{\alpha})\mid \alpha\in\mathbb{N}^n\setminus\{(0,\cdots, 0)\}\}$, where $D_H: \calr\longrightarrow \calr$ is a linear mapping defined by $D_H(x^{\alpha})=\sum\limits_{i=1}^n\sigma(i)\partial_i(x^{\alpha})\partial_{i^{\prime}}$ with
		\begin{equation*}
			\sigma(i)=\begin{cases}
				1, &\text{if}\,\, 1\leq i\leq r,\cr
				-1, &\text{if}\,\, r+1\leq i\leq n,
			\end{cases}
		\end{equation*}
		and
		\begin{equation*}
			i^{\prime}=\begin{cases}
				i+r, &\text{if}\,\, 1\leq i\leq r,\cr
				i-r, &\text{if}\,\, r+1\leq i\leq n.
			\end{cases}
		\end{equation*}
		Since the 2-form  $\sum_{i=1}^r dx_i\wedge dx_{i+r}$ can be regarded as an operator of degree 2, the algebra $H(n)$ inherits the $\mathbb{Z}$-grading of $W(n)$.
		
		In the following, let $\ggg=X(n)$, $X\in\{W,S,H\}$. Then $\ggg$ has a $\mathbb{Z}$-grading $\ggg=\bigoplus\limits_{i=-1}^{\infty}\ggg_{i}$, where $\ggg_{ i}=\ggg\cap W(n)_{i}$ for $i\geq -1$. Let $\ggg_{\geq i}=\bigoplus_{j\geq i}\ggg_{j}$. We then have the following $\mathbb{Z}$-filtration of $\ggg$: $$\ggg=\ggg_{\geq-1}\supset \ggg_{\geq0}\supset\ggg_{\geq1}\cdots.$$
		It should be noted that
		\begin{equation}\label{grading 0}
			\ggg_{0}\cong\begin{cases}
				\ggg\lll(n),&\text{if}\,\, \ggg=W(n),\cr
				\sss\lll(n),&\text{if}\,\, \ggg=S(n),\cr
				\sss\ppp(n),&\text{if}\,\, \ggg=H(n).
			\end{cases}
		\end{equation}
		We have a triangular decomposition $\ggg_{0}=\nnn^-\oplus\hhh\oplus\nnn^+$, where
		\begin{equation*}
			\nnn^{-}=\begin{cases}
				\text{span}_{\mathbb{F}}\{x_i\partial_j\mid 1\leq j<i\leq n\}, &\text{if}\,\, \ggg=W(n), S(n),\cr
				\text{span}_{\mathbb{F}}\{x_i\partial_j-x_{j+r}\partial_{i+r}, x_{s+r}\partial_{t}+x_{t+r}\partial_{s}\mid \cr
				\quad\quad\quad 1\leq j<i\leq r, 1\leq s\leq t\leq r\}, &\text{if}\,\, \ggg=H(2r),
			\end{cases}
		\end{equation*}
		
		\begin{equation*}
			\hhh=\begin{cases}
				\text{span}_{\mathbb{F}}\{x_i\partial_i\mid 1\leq i\leq n\}, &\text{if}\,\, \ggg=W(n),\cr
				\text{span}_{\mathbb{F}}\{x_i\partial_i-x_j\partial_j\mid 1\leq i<j\leq n\}, &\text{if}\,\, \ggg=S(n),\cr
				\text{span}_{\mathbb{F}}\{x_i\partial_i-x_{i+r}\partial_{i+r}\mid 1\leq i\leq r\}, &\text{if}\,\, \ggg=H(2r),
			\end{cases}
		\end{equation*}
		and
		\begin{equation*}
			\nnn^{+}=\begin{cases}
				\text{span}_{\mathbb{F}}\{x_i\partial_j\mid 1\leq i<j\leq n\}, &\text{if}\,\, \ggg=W(n), S(n),\cr
				\text{span}_{\mathbb{F}}\{x_i\partial_j-x_{j+r}\partial_{i+r}, x_s\partial_{t+r}+x_t\partial_{s+r}\mid\cr
				\quad\quad\quad 1\leq i<j\leq r, 1\leq s\leq t\leq r\}, &\text{if}\,\, \ggg=H(2r).
			\end{cases}
		\end{equation*}
		The negative root system associated with $\nnn^-$  is denoted by $\Phi^-$.
		Let $\sfp=\ggg_{\leq 0}:=\ggg_{{-1}}\oplus\ggg_{0}$, and $U(\sfp), U(\ggg)$ be the universal enveloping algebras of $\sfp$ and $\ggg,$ respectively. The $\mathbb{Z}$-grading on $\ggg$ (resp. $\sfp$) induces a natural $\mathbb{Z}$-grading on $U(\ggg)$ (resp. $U(\sfp)$).

\subsection{The category \texorpdfstring{$\mathcal{O}$}{} } 
The following notion is an analogy of the BGG category for complex finite dimensional semi-simple Lie algebras.
		
		\begin{defn}\label{category O}
			Denote by $\mathcal{O}$ the category, { whose objects $M$ are $U(\ggg)$-modules}  with the following three properties satisfied.
			\begin{itemize}
				\item[(1)] $M$ is an admissible $\mathbb{Z}$-graded $\ggg$-module, i.e., $M=\bigoplus\limits_{i\in\mathbb{Z}} M_{i}$ with $\dim M_{i}<+\infty$, and
				$\ggg_{i}M_{j}\subseteq M_{i+j},\forall\,i,j$.
				\item[(2)] $M$ is locally finite as a $\sfp$-module. Here $\sfp=\ggg_{\leq 0}:=\ggg_{{-1}}\oplus\ggg_{0}$  is defined in \S\,\ref{CartanType}.
				{ \item[(3)] $M$ is $\hhh$-semisimple, i.e., $M$ is a weight module: $M=\bigoplus_{\lambda\in\hhh^*}M_{\lambda}$.}
			\end{itemize}
			The morphisms in $\mathcal{O}$ are the $\ggg$-module morphisms that respect the $\mathbb{Z}$-grading, i.e.,
			$$\Hom_{\mathcal{O}}(M, N)=\{f\in\Hom_{U(\ggg)}(M, N)\mid f(M_{i})\subseteq N_{i},\,\forall\,i\in\mathbb{Z}\},\,\forall\,M,N\in\mathcal{O}.$$
		\end{defn}

We have the following observation on $\co$.
\begin{lem}\label{lem: depth}
Let $M$ be an object in $\co$, satisfying the property that is finitely-generated over $U(\ggg)$. Then there is a unique integer $d$ such that $M=\sum_{i\geq d}M_d$ with $M_d\ne0$.
\end{lem}
		
\begin{proof} By assumption  that $M$ is finitely generated over $U(\ggg)$, and locally finite over $U(\sfp)$, it is readily known that there is a finite-dimensional  $\mathbb{Z}$-graded $U(\sfp)$-module $V$ such that $V$ is a generating  space over $U(\ggg)$.
 Assume $V=\sum_{j=1}^tV_{i_j}$ with  $i_1<\cdots<i_j<\cdots <i_t$. From $M=U(\ggg)V=U(\ggg_{\geq1})V$  it follows that $M=\sum_{i\geq i_1}M_i$. In particular, $i_1$ is  the  exactly desired integer $d$.
		\end{proof}
		
		The integer $d$ in lemma \ref{lem: depth} is called the depth of $M$, often written as $d(M)$. Sometimes, the above $M_d$  is called the depth space of $M$.

		\subsubsection{Shift functors by shifting depthes}
		Set $$\mathcal{O}_{\geq d}:=\{M\in\mathcal{O}\mid M=\sum\limits_{i\geq d}M_{i}\},$$ which consists of objects admitting depths not smaller than $d$.

	Consider the shift functor $T_{d,d'}:\mathcal{O}_{\geq d}\longrightarrow\mathcal{O}_{\geq d'}$ which by definition  $M=\sum_{k}M_k\in \mathcal{O}_{\geq d}$ is assigned  { to $M[d-d']\in\mathcal{O}_{\geq d'}$ where $M[d-d']$ denotes the same underling space as $M$, but $M[d-d']_k=M_{k+d-d'}$. }Then the functor $T_{d,d'}$ induces a category equivalence. With the shift functors, we can focus our concern on $\mathcal{O}_{\geq 0}$ (or $\mathcal{O}_{\geq d}$ for some specific depth $d$) when we make arguments on module structures.
		
		

		\subsection{Standard modules}
		\subsubsection{}\label{basic notations}
		Keep notations as before, in particular, $\ggg=\bigoplus\limits_{i=-1}^{\infty}\ggg_{i}$ is one of the Lie algebras of vector fields $W(n), S(n)$ and $H(n)$, and $\hhh$ is the standard Cartan subalgebra of $\ggg_{0}$ (recall $\ggg_{0}\cong \mathfrak{gl}(n)$ for $W(n)$, $\mathfrak{sl}(n)$ for $S(n)$ and $\mathfrak{sp}(n)$ for $H(n)$ under the isomorphism correspondence $W_{0}\rightarrow \mathfrak{gl}(n)$ with $x_i\partial_j\mapsto E_{ij}$). Denote by $\epsilon_i$ the linear function on $\sum_{j=1}^n\bbf x_j\partial_j$ via defining
		$\epsilon_i(x_j\partial_j)=\delta_{ij}$ for $1\leq i,j\leq n$. In the natural sense, we  identify the unit function $\epsilon_i$ with $(\delta_{1i},\cdots, \delta_{ni})$ for $1\leq i\leq n$.  With those unit linear functions, we can express the weight functions that we need for the arguments on $\ggg_{0}$-modules in the sequent.
		Let $\Lambda^+$ be the set of dominant integral weights relative to the standard Borel subalgebra $\mathfrak{b}:=\hhh+\nnn^+$ of $\ggg_{0}$. Then finite dimensional irreducible $\ggg_{0}$-modules are parameterized by ${\Lambda^+}\times\mathbb{Z}$. For any $\lambda\in {\Lambda^+}$, let ${}^dL^0(\lambda)$ be the simple $\ggg_{0}$-module concentrated in a single degree $d$ with the highest weight $\lambda$.  Set ${}^d\Delta(\lambda)=U(\ggg)\otimes_{U(\sfp)}{}^dL^0(\lambda)$,  where ${}^dL^0(\lambda)$ is regarded as a $\sfp$-module with trivial $\ggg_{{-1}}$-action. Then ${}^d\Delta(\lambda)$ is a standard module of depth $d$, and
		$\{{}^d\Delta(\lambda)\mid \lambda\in {\Lambda^+},d\in\mathbb{Z}\}$ constitutes a class of so-called standard modules of depth $d$ for $\ggg$ in the usual sense. We have the following result.
		\begin{lem}\label{lem1}
			Let $\lambda\in{\Lambda^+}, d\in\mathbb{Z}$. The following statements hold.
			\begin{itemize}
				\item[(1)] The standard module ${}^d\Delta(\lambda)$ is an object in $\mathcal{O}$.
				\item[(2)] The standard module ${}^d\Delta(\lambda)$ has a unique irreducible quotient, denoted by ${}^dL(\lambda)$.
				\item[(3)] The iso-classes of irreducible modules in $\mathcal{O}$ are parameterized by $\Omega:={\Lambda^+}\times \mathbb{Z}$. More precisely, each simple module $S$ in $\mathcal{O}$ is of the form $L(\mu)$ for some $\mu\in\Lambda^+$ with depth $d$.
			\end{itemize}
		\end{lem}

		\begin{rem} When $d=0$, we usually write ${}^0\Delta(\lambda)$ (resp. ${}^0L^0(\lambda)$, ${}^0L(\lambda)$) as  $\Delta(\lambda)$ (resp. $L^0(\lambda)$, $L(\lambda)$) for brevity.
		\end{rem}

		\subsection{Co-standard modules and their prolonging realization}
		Keep the same notations as in the previous sections.
		\subsubsection{Co-standard modules} Let $\lambda\in{\Lambda^+}$. Define the co-standard $\ggg$-module corresponding to $\lambda$ as { $$\nabla(\lambda):=\mathcal{H}om_{U(\ggg_{\geq0})}(U(
		\ggg), L^0(\lambda)),$$
		where $L^0(\lambda)$ is regarded as a $\ggg_{\geq 0}$-module with trivial $\ggg_{\geq 1}$-action. Here for two $\mathbb{Z}$-graded $U(\ggg_{\geq0})$-modules $M$ and $N$,  $\mathcal{H}om_{U(\ggg_{\geq0})}(M,N)=\bigoplus_{i\in\mathbb{Z}}\mathcal{H}om_{U(\ggg_{\geq0})}(M,N)_i$ denotes the the $\mathbb{Z}$-graded vector space with homogeneous components
$$\mathcal{H}om_{U(\ggg_{\geq0})}(M,N)_i=\{f\in Hom_{U(\ggg_{\geq0})}(M, N)\mid f(M_j)\subseteq N_{i+j}, \,\forall\, i\in\mathbb{Z}\}.$$}
		Then it is readily known that $\nabla(\lambda)\in\mathcal{O}_{\geq 0}$, parallel to \cite{Soe}. We have the following result.
		\begin{lem}\label{pro-inj}
			Let $\lambda,\mu\in{\Lambda^+}$, then the following statements hold.
			\begin{itemize}
				\item[(1)] $L(\lambda)$ admits a projective cover $\Delta(\lambda)$ in $\mathcal{O}_{\geq 0}$.
				\item[(2)] $L(\lambda)$ admits an injective hull $\nabla(\lambda)$ in $\mathcal{O}_{\geq 0}$.
				\item[(3)] $\Hom_{\mathcal{O}_{\geq 0}}(\Delta(\lambda), \nabla(\mu))=0$ if $\lambda\neq \mu$.
				\item[(4)] $\Ext_{\mathcal{O}_{\geq 0}}^1(\Delta(\lambda), \nabla(\mu))=0$ for any $\lambda,\mu$.
			\end{itemize}
		\end{lem}

		\subsubsection{Prolongation realization}\label{sec: prolong real} In this subsection, we recall  a kind of realization of co-standard modules $\nabla(\lambda)$ for $\lambda\in{\Lambda^+}$ via prolonging irreducible $\textsf{P}$-module $L^0(\lambda)$ which is an irreducible $\ggg_0$-module with $\ggg_{-1}$-trivial action. Set $\mathcal{V}(\lambda)=\calr\otimes L^0(\lambda)$ for $\lambda\in{\Lambda^+}$. It follows from \cite[Theorem 2.1]{Skr}  that we can endow with a $W(n)$-module structure $\rho_{_{W(n)}}$ on $\mathcal{V}(\lambda)$ via
		\begin{equation}\label{W(n)-module structure}
			\rho_{_{W(n)}}(\sum\limits_{i=1}^n f_i\partial_i)(g\otimes v)=\sum\limits_{i=1}^n f_i(\partial_i(g))\otimes v+\sum\limits_{i=1}^n\sum\limits_{j=1}^n (\partial_j(f_i))g\otimes \xi(x_j\partial_i)v
		\end{equation}
		for any $f_i,g\in \calr, v\in L^0(\lambda)$,
		where $\xi$ is the representation of $W(n)_{0}$ on $L^0(\lambda)$. Furthermore, it is a routine to check that we have a $\ggg$-module structure
		on $\mathcal{V}(\lambda)$
		via:
		\begin{eqnarray}\label{S(n)-module}
			\rho_{\ggg}(D_{kl}(x^{\alpha})) (g\otimes v)&=&(D_{kl}(x^{\alpha}))(g)\otimes v+\alpha(k)
			\alpha(l)x^{\alpha-\epsilon_k-\epsilon_l}g
			\otimes\xi(x_k\partial_k-x_l\partial_l)v\nonumber
			\\
			&&+\sum\limits_{\stackrel{j=1}{j\neq k}}^n
			\alpha(l)(\alpha(j)-\delta_{jl})x^{\alpha-\epsilon_j-\epsilon_l}g\otimes\xi(x_j\partial_k)v\nonumber\\
			&&-\sum\limits_{\stackrel{j=1}{j\neq l}}^n
			\alpha(k)(\alpha(j)-\delta_{jk})
			x^{\alpha-\epsilon_j-\epsilon_k}g\otimes\xi(x_j\partial_l)v,
		\end{eqnarray}
		and
				\begin{eqnarray}\label{H(n)-module}
			\rho_{\ggg}(D_H(x^{\alpha})) (g\otimes v)&=&(D_H(x^{\alpha}))(g)\otimes v+\sum\limits_{j=1}^{2r}\sigma(j)\alpha(j)(\alpha(j)-1)x^{\alpha-2\epsilon_j}g\otimes \xi (x_{j}\partial_{j^{\prime}})v\nonumber   \\
			&&+\sum\limits_{1\leq j<k\leq r}\alpha(j)\alpha(k)x^{\alpha-\epsilon_j-\epsilon_k}g\otimes\xi(x_k\partial_{j^{\prime}}+x_j\partial_{k^{\prime}})v\nonumber\\
			&&-\sum\limits_{k=1}^r\sum\limits_{j=r+1}^{2r}\alpha(j)\alpha(k)x^{\alpha-\epsilon_j-\epsilon_k}g\otimes\xi(x_k\partial_{j^{\prime}}-
			x_j\partial_{k^{\prime}})v\nonumber\\
			&&-\sum\limits_{r+1\leq j<k\leq 2r}\alpha(j)\alpha(k)x^{\alpha-\epsilon_j-\epsilon_k}g
			\otimes\xi(x_k\partial_{j^{\prime}}+x_j\partial_{k^{\prime}})v
		\end{eqnarray}
		for $\ggg=S(n), H(n),$ respectively. Here $\alpha\in\mathbb{N}^n, 1\leq k<l\leq n, g\in \calr, v\in  L^0(\lambda)$,
		$\xi$ is the representation of $\ggg_{0}$ on $L^0(\lambda)$.

		\begin{rem}  Professor Guangyu Shen found the prolongation of $\ggg_0$-modules in \cite{Shen1}, by constructing ``mixed product".   A conceptual account for his construction  was provided in \cite[Theorem 1.2]{Shen1}.
It is worth mentioning that we adopt here Skryabin's presentation in \cite{Skr} for type $W$,  and in \cite{YS1, YS2} for types $S$ and $H$.

These modules were also constructed by Larsson in \cite{Lar}.
		\end{rem}		
The following result asserts that the co-standard
		$\ggg$-module $\nabla(\lambda)$ is isomorphic to $\mathcal{V}(\lambda)$ for $\ggg=X(n)$, $X\in\{W,S,H\}$.
		
		\begin{prop}\label{co-standard iso}
			Keep the notations as above, then $\nabla(\lambda)\cong \mathcal{V}(\lambda)$ as $U(\ggg)$-modules.
		\end{prop}

		\subsubsection{}  Recall the notations $\epsilon_i$ in \S\ref{basic notations} for the unit linear functions on span${}_\bbf\{x_i\partial_i\mid i=1,\ldots,n\}\cong\hhh$. We have the following definition of exceptional weights for further use.
		\begin{defn}
			Let $\ggg=X(n), X\in\{W,S,H\}$, be a Lie algebra of vector fields. Set $\omega_0=0$ and $\omega_k=\epsilon_{1}+\epsilon_{2}+\cdots+\epsilon_k$			
			for $1\leq k\leq n^{\prime}$, where
			\begin{equation*}
				n^{\prime}=\begin{cases}
				       n-1,&\text{if}\,\,X=W,S,\cr
					\frac{n}{2},&\text{if}\,\,X=H.
				\end{cases}
			\end{equation*}
			These $\omega_k\,(0\leq k\leq n^{\prime})$ are called exceptional weights. The corresponding simple $\ggg$-modules $L(\omega_k)$ ($0\leq k\leq n^{\prime}$) are called exceptional $\ggg$-modules.
		\end{defn}
		
		The following result is due to A. Rudakov and G. Shen.
		
		\begin{prop}(\cite[Theorem 13.7, and Corollaries 13.8-13.9]{Rud}, \cite[Theorem 4.8]{Rud2} and \cite[Theorem 2.4]{Shen})\label{known result1}
			Let $\ggg=W(n)$ or $S(n)$. Then the following statements hold.
			\begin{itemize}
				\item[(1)] If $\lambda\in{\Lambda^+}$ is not exceptional, then $\mathcal{V}(\lambda)$ is a simple $\ggg$-module.
				\item[(2)] The following sequence
\begin{align*}
0\longrightarrow\mathcal{V}(\omega_0)\xrightarrow{\,\,\,d_0\,\,\,}\mathcal{V}(\omega_1)\xrightarrow{\,\,\,d_1\,\,\,}\cdots\cdots \mathcal{V}(\omega_k)\xrightarrow{\,\,\,d_k\,\,\,}
\mathcal{V}(\omega_{k+1})\xrightarrow{d_{k+1}}\cdots\longrightarrow \cr
\cdots\longrightarrow \mathcal{V}(\omega_{n-1})\xrightarrow{d_{n-1}} \mathcal{V}(\omega_{n})\longrightarrow 0
\end{align*}
		
			is exact, where
			\begin{eqnarray*}
				d_k:\,\mathcal{V}(\omega_k)&\longrightarrow &\mathcal{V}(\omega_{k+1})\\
				x^{\alpha}\otimes (v_{j_1}\wedge\cdots\wedge v_{j_k})&\longmapsto&\sum\limits_{i=1}^n\partial_i(x^{\alpha})\otimes (v_{j_1}\wedge\cdots\wedge v_{j_k}\wedge v_i),\\
& &\forall\,\alpha\in {\mathbb{N}}^n, 1\leq j_1<\cdots < j_k\leq n.
			\end{eqnarray*}
			\item[(3)] For $0\leq k\leq n-1$, $\mathcal{V}(\omega_k)$ contains two composition factors $L(\omega_k)$ and $L(\omega_{k+1})$ with free multiplicity. Moreover, $\mathcal{V}(\omega_n)\cong L(\omega_n)$.
			
			\end{itemize}	
		\end{prop}
		
		\begin{prop}(\cite[Theorem 5.10]{Rud2} and \cite[Theorem 2.5]{Shen})\label{known result2}
			Let $\ggg=H(n)$, $n=2r$. Then the following statements hold.
			\begin{itemize}
				\item[(1)] If $\lambda\in{\Lambda^+}$ is not exceptional, then $\mathcal{V}(\lambda)$ is a simple $\ggg$-module.
				\item[(2)] The composition factors of $\mathcal{V}(\omega_k)$ are $L(\omega_{k-1}), L(\omega_{k})$ and $L(\omega_{k+1})$ with
				$[\mathcal{V}(\omega_k):L(\omega_{k-1})]=[\mathcal{V}(\omega_k):L(\omega_{k+1})]=1$ and $[\mathcal{V}(\omega_k):L(\omega_{k})]=2$, $0\leq k\leq r$, where we appoint  $L(\omega_{-1})=0$, $\omega_{r+1}=\epsilon_{1}+\epsilon_{2}+\cdots+\epsilon_{r+1}$.
			\end{itemize}
		\end{prop}
		
		\begin{rem} There is a modular version of Propositions \ref{known result1}, \ref{known result2} (cf. \cite[Theorems 2.1, 2.2, 2.3]{Shen}).
		\end{rem}

		\section{Lie-Cartan modules}

				\subsection{{Definition of  a Lie-Cartan module}} 
        \begin{defn}\label{def: lc mod}
 Call $M\in\co$ a Lie-Cartan module if $M$ is an 
					$\calr$-module with action $\theta$, and a $\ggg$-module with action $\rho$. Both are compatible  in the following sense.
					\begin{itemize}
						\item[(LC-1)] 
						$[\rho(X),\theta(f)]=\theta(X(f))\text{ for any }f\in \calr \text{ and } X\in \ggg$.
						\item[(LC-2)] The degrees in $\calr$ are compatible with the gradings of $M$ in $\co$. This is to say,  $\theta(\calr_i)M_j\subset M_{i+j}$.
					\end{itemize}
					We will often indicate a Lie-Cartan module by the triple $(M,\rho,\theta)$, and call the pair $(\rho,\theta)$ its structure mappings.
				\end{defn}
				
				\begin{rem} Under his supervisor Bin Shu's suggestions and instructions, Lianqing Geng introduced Lie-Cartan modules and studied some basic properties  in his thesis at East China Normal University  \cite{Geng} in 2022-2023.
				\end{rem}

		\begin{rem} \label{ex: Lie-Cartan containing}
			\begin{itemize}
				\item[(1)]
				It is a routine to check that $\cv(\lambda)$ is a Lie-Cartan module with free $\calr$-module structure and $\ggg$-module structure defined by (\ref{W(n)-module structure}), (\ref{S(n)-module}) and (\ref{H(n)-module}).
				
				\item[(2)]
				It is worth mentioning that the Lie-Cartan module $\cv(\lambda)$ contains the irreducible $\ggg$-submodule $L(\lambda)$ by Lemma \ref{pro-inj} and Proposition \ref{co-standard iso}. This fact will be used later.
			\end{itemize}
		\end{rem}

\subsection{The Lie-Cartan modules \texorpdfstring{$\cv(\lambda)$}{} }

	\begin{lem}\label{lem: basic prop} Let $(M,\rho,\theta)$ be a Lie-Cartan module with $M=\sum_{k\geq 0}M_k$ of depth $0$.
		Then the homogeneous subspace $M_k$ contains the subspace $S_k:=\theta(\calr_k)M_0$. Furthermore, $S_k$ is free over $\calr_k$ of rank $m:=\dim M_0$ in the sense that
				\begin{itemize}
					\item[(1.1)] $\theta(\calr_k)M_0=\sum_{i=1}^m \theta(\calr_k)v_i$ for a basis $\{v_i\mid i=1,\ldots,m\}$ of $M_0$, and
					\item[(1.2)] $\sum_{i=1}^m \theta(f_i)v_i=0$ with $f_i\in\calr_i$ implies all $f_i=0$.
				\end{itemize}
	\end{lem}
		
\begin{proof}
The statements of the first part and of (1.1) are clearly true.  We only need to show (1.2).  We will do that by induction on the degree $k$.
			
If $k=0$, (1.2) is obviously true.  Suppose $k>0$ and the statement (1.2) for degree $k-1$ is  already proved to be true.
Before the arguments, let us keep in mind that
			\begin{align}\label{eq: depth}
				\rho(\partial_i)M_0=0, \,\forall\,1\leq i\leq n
			\end{align}
			because the depth of $M$  is $0$.
			For any  $f_i\in \calr_k$   $(i=1,2,\ldots, m)$, as long as
			$$\sum_{i=1}^m \theta(f_i)v_i=0,$$
			then for any $q\in \{1,2,\ldots,m\}$ we have
			\begin{align*}
				0&=\rho(\partial_q)\sum_{i=1}^m \theta(f_i)v_i\cr
				&=\sum_{i=1}^m (\theta(f_i)(\rho(\partial_q)v_i)+\theta(\partial_q(f_i))v_i)\cr
				&=\sum_{i=1}^m\theta(\partial_q(f_i))v_i.
			\end{align*}
			The last equation is due to (\ref{eq: depth}).
			Thus $\sum_{i=1}^m\theta(\partial_q(f_i))v_i=0$. Note that $\theta(\partial_q(f_i))\in \calr_{k-1}$.
			By the inductive hypothesis, we already have $\partial_q(f_i)=0$ for all $f_i$, $i=1,\ldots,m$. When $q$ runs through $\{1,\ldots, n\}$, it is deduced that all $\partial_q(f_i)=0$ for $q=1,\ldots,n$. Hence all $f_i$ are constants,  and fall in $\calr_0$. This implies that $f_i\in\calr_0\cap\calr_k=0$ for any $i$ by the assumption that $k>0$. The proof is completed.
		\end{proof}
		
			Furthermore, from the above lemma, we have
		\begin{prop}\label{prop: v lambda irr} Any $\cv(\lambda)$ is an irreducible Lie-Cartan module.
		\end{prop}
		
		\begin{proof} It follows directly from the definition of $\cv(\lambda)$ and  Lemma \ref{lem: basic prop}.
		\end{proof}
		
	Lemma \ref{lem: basic prop} also  implies that $\theta(\calr)M_0$ is isomorphic to $\calr\otimes M_0$ as an $\calr$-module.

			\begin{prop}\label{prop: unique lc v} Let $\ggg=X(n)$ for $X\in\{W,S,H\}$. There is only one structure of Lie-Cartan module on $\cv(\lambda)$ for any $\lambda\in \Lambda^+$, up to equivalences.
			\end{prop}
			
			\begin{proof} Suppose that $\calv$ is endowed with a Lie-Cartan module structure with structure mappings $(\rho',\theta')$.
 We will show that $(\rho',\theta')$ is not other than the one defined by  (\ref{W(n)-module structure}) when $\ggg=W(n)$, (\ref{S(n)-module}) when $\ggg=S(n)$, and (\ref{H(n)-module}) when $\ggg=H(n)$.
 Note that $\cv(\lambda)$ is a free $\calr$ module by Lemma \ref{lem: basic prop}. Hence,  it is sufficient to show that the pair $(\rho',\theta')$ satisfies the above equations at $1 \otimes v$ for $v\in \cv(\lambda)$.
 We will verify  this by induction on each $\mathfrak{g}_i$-action on $1 \otimes v$.
The verification will proceed in different cases.

(1) $\ggg=W(n)$.
						
There is nothing to prove for $i =-1 $ and $i=0$. Assume that the induction principal holds for $i \leq m-2$ and some $m \geq 2$. Now let $i =m-1$. Since the action of $\ggg$ is compatible with the grading of $\cv(\lambda)$, we can assume that
			\begin{align}\label{align2.4.1}
							\rho'(x_{i_1}\dots x_{i_m}\partial_{k}) (1 \otimes v)= \displaystyle{\sum_{1\leq s_1,\dots,s_{m-1}\leq n}} x_{s_1}\dots x_{s_{m-1}} \otimes v_{s_1\dots s_{m-1}}.
			\end{align}
						
						Now consider the action of $\partial_{q_1}\dots \partial_{q_{m-1}}$ on the both sides of (\ref{align2.4.1}), for $1\leq q_i \leq n$, $1 \leq i \leq m-1$. Then the right hand side (RHS for short) of (\ref{align2.4.1}) becomes $1 \otimes v_{q_1\dots q_{m-1}}	$. On the other hand,
						
						\begin{equation*}
							\text{LHS of (\ref{align2.4.1})}=\begin{cases}
								1 \otimes \xi(x_{i_r}\partial_k)\cdot v,&\text{if}\,\, \{q_1,\dots ,q_{m-1}\} = \{i_1, \dots, i_{m}\}\setminus \{i_r\},\cr
								0,&\text{otherwise}.
							\end{cases}
						\end{equation*}
						Therefore comparing LHS and RHS, we have $1 \otimes v_{i_1 \dots \widehat{i_r} \dots {i_m}} = 1\otimes \xi(x_{i_r}\partial_k)\cdot v$,
						and all other $1 \otimes v_{s_1\dots s_{m-1}}$ = $0$, where $\widehat{i_r}$ represents that $i_r$ does not appear in that expression. Hence, we have
						{
						 $$\rho'(x_{i_1}\dots x_{i_m}\partial_{k}) (1 \otimes v)= \displaystyle{\sum_{r=1}^{m}}\partial_{i_r}(x_{i_1}\dots x_{i_m})\otimes \xi(x_{i_r}\partial_{k})\cdot v=\rho(x_{i_1}\dots x_{i_m}\partial_{k}) (1 \otimes v)$$ }
					 as desired.
						
						(2) $\ggg=S(n), H(n)$.
						
						Note that the action of $S(n), H(n)$ is the restriction of that of $W(n)$. Similar arguments as in the case $W(n)$ yield the desired action of $S(n), H(n)$.
						
						This completes the proof.				
			\end{proof}

		\subsection{}
		We can consider the category $\lc$ of Lie-Cartan modules whose objects come from all Lie-Cartan modules. The homomorphism space between any $M$ and $N$ in $\lc$ is defined by
		$$ \Hom_{\lc}(M,N)=\{\varphi\in \Hom_\co(M,N)\mid \varphi(\theta_M(f)m)=\theta_N(f)\varphi(m),\forall m\in M, f\in\calr\}$$
		where $\theta_M$ and $\theta_N$ denote the $\calr$-module structure on $M$ and $N$, respectively.  Similarly, we can talk about the subcategories $\lc_{d}$ and $\lc_{\geq d}$ of $\lc$.  Note that there is a category equivalence by the shift of depth between  $\lc_d$ and $\lc_0$. So the arguments in the sequel can be only focused on $\lc_0$. The following fact is clear.
		\begin{lem} The category $\lc$ is abelian.
		\end{lem}
%
%
		%
		
		 In the subsequent arguments, we often call morphisms in $\lc$  Lie-Cartan module homomorphisms, and call sub-objects of an object $M$ in $\lc$  Lie-Cartan submodules of $M$.

\subsection{Finite-depth subcategory \texorpdfstring{$\lc^+$}{} }		

Furthermore, we introduce a more natural subcategory.
\begin{defn}\label{defn: positive lc} Define the subcategory $\lc^+$ of $\lc$ whose objects have  finite depths. This is to say, for any object $M\in \lc^+$ there is an integer $d\in \bbz$ such that $M=\sum_{i\geq d}M_i$.
\end{defn}

Obviously, $\lc^+$ is a full subcategory of $\lc.$  Denote by $\lc_f$ the subcategory of $\lc$ whose objects are finitely generated over $U(\ggg)$. Then $\lc_f$ becomes a full subcategory of $\lc^+$ by lemma \ref{lem: depth}.

	\subsection{} Similar to Definition \ref{defn: positive lc}, we can define the subcategory $\co^+$ of $\co$ whose objects admit finite depth. The following result gives a way to obtain a Lie-Cartan module from a module in $\co^+$.
	
		\begin{lem}\label{lem: M_R}
 Let { $M\in \co^+$. Then $M_{\calr}:=\calr\otimes_\bbf M$} is endowed with a Lie-Cartan module structure by usual $\ggg$-action $\rho$ and   regular (left multiplication) $R$-action $\theta$ on the tensor module.
\end{lem}
		
		\begin{proof} By definition, $M_\calr$ obviously becomes an $\calr$-module. The $\ggg$-module structure mapping $\rho$ on $M_\calr$ is by definition as below
\begin{align}\label{eq: define g-module}
\rho(X) (f\otimes v)=X(f)\otimes v+f\otimes X. v,\,\,\forall X\in \ggg, f\in \calr, v\in M.
\end{align}
		By a direct verification,  it is easily known that (\ref{eq: define g-module}) defines a $\ggg$-module structure on $M_\calr$, and  $M_\calr$ is still in { $\co_{\geq d}$.}

Now check that $M_\calr$ becomes a Lie-Cartan module. For any $X\in\ggg$, $f\in\calr$ and $r\otimes v\in M_\calr$, we have
\begin{align*}
\rho(X)\theta(f)(r\otimes v)=&\rho(X)(fr\otimes v)\cr
=& X(fr)\otimes v+fr\otimes X.v  \cr
=&X(f).r\otimes v+fX(r)\otimes v+fr\otimes X.v  \cr
=&\theta(X(f))(r\otimes v)+\theta(f)\rho(X)(r\otimes v).
\end{align*}
This completes the proof.
\end{proof}

\begin{rems}
(1) When $M=\bbf$ is the trivial $U(\ggg)$-module, it is readily known that $\bbf_{\calr}\cong\cv(0)$ by (\ref{eq: define g-module}).
				
(2) For the $U(\ggg)$-module $\cv(\lambda)$ with $\lambda\in {\Lambda^+}$, it follows from Proposition \ref{co-standard iso} that  $\cv(\lambda)\cong
\nabla(\lambda)\in\mathcal{O}_{\geq 0}$. Then the following mapping
$$	\aligned \psi:\quad \cv(\lambda)_{\calr}:=\calr\otimes_\bbf \cv(\lambda) &\longrightarrow
				\cv(\lambda)\cr
				r_1\otimes (r_2\otimes v) &\longmapsto r_1r_2\otimes v,\,\,\, \forall\, r_1, r_2\in\calr, v\in L^0(\lambda)
				\endaligned	$$
gives an epimorphism from $\cv(\lambda)_{\calr}$ to $\cv(\lambda)$.

(3) Naturally, we can regard $\cv(\lambda)$ and $\Delta(\lambda)_\calr$ as co-standard objects  and standard objects in $\lc$. More explanation on this can be read out  from {Corollary \ref{cor: ext exp}(2).}

(4) { From now on, we often denote by $V_\calr$ the base change of an $\bbf$-vector space $V$ through $\calr$, i.e.
$$ V_\calr:= \calr\otimes_\bbf V.$$
 }
\end{rems}

			\subsection{The depth functor}\label{sec: dep fun}
			Suppose $M\in\lc^+$ with depth $d$. Then we write $M=\sum_{i\geq d}M_i$. Set $\mathscr{D}(M)$ to be the  Lie-Cartan submodule of $M$ generated by $M_d$. 
In the following, we set  $\text{Ann}_{\ggg_{-1}}(M)=\{v\in M\mid \ggg_{-1}.M=0\}$.

			
			\begin{lem} The above $\mathscr{D}$ defines a functor from $\lc^+$ to $\lc^+$.
			\end{lem}
			
			\begin{proof} It immediately follows from the definitions of $\lc^+$ and $\mathscr{D}$.
			\end{proof}

			\subsection{Radicals in \texorpdfstring{$\lc^+$}{} }\label{sec: radical}
			\begin{lem}\label{lem: radical} Let $M$ be a Lie-Cartan module of depth $d=d(M)$. The following statements hold.
				\begin{itemize}
					\item[(1)] There is a unique maximal Lie-Cartan  submodule $N$ satisfying that its $d$-th homogeneous space is zero and it contains any  submodule $N'$ of $M$ with $d(N')>d$ if $N'$ exists.

\item[(2)]   $M=\mathscr{D}(M)+ N$, where the functor $\mathscr{D}$ is  introduced in \S\ref{sec: dep fun}.
					
					\item[(3)] $M/N$ is semisimple in $\lc$, which is isomorphic to a direct sum of $\cv(\lambda_i)$ for some $\lambda_i\in \Lambda^+$, $i=1,\ldots,t$.
					
				\end{itemize}
			\end{lem}
			
			\begin{proof}
{ (1) The sum of any two Lie-Cartan submodules with depth bigger than $d$ is still a Lie-Cartan submodule with depth bigger than $d$. Then the unique maximal Lie-Cartan  submodule $N$ is just the sum of all Lie-Cartan submodules with depth bigger than $d$. (1) follows.

(2) and (3) We only need to show the statements under the following assumption.

 (*) $M$ does not contain any nonzero Lie-Cartan submodule of depth bigger than $d$.

Indeed, suppose (2) and (3) hold under the assumption (*). Then in general case, note that $M/N$ satisfies the assumption (*), (3) holds and  we have $$(\mathscr{D}(M)+N)/N=\mathscr{D}(M/N)=M/N.$$
This implies that $M=\mathscr{D}(M)+N$, i.e, (2) holds.

In the subsequent discussion, we always suppose that the assumption (*) holds. We have the following claims.

\textbf{Claim 1:} $\cala(M):=\text{Ann}_{\ggg_{-1}}(M)$ must coincide with $M_d$.

Obviously, $M_d\subset \cala(M)$. We show $M_d\supset \cala(M)$ by contradiction. Suppose there exits nonzero vector $v\in \cala(M)\backslash M_d$. We can further suppose $v\in M_{j}$ with $j>d$. Then $U(\ggg)v=U(\ggg_{\geq0})v\subset M_{\geq j}$. Note that the action of $\calr$ preserves ascending of gradings. So  the Lie-Cartan submodule of $M$ generated by $v$ admits the depth bigger than $d$, which contradicts the assumption in the beginning of the arguments.

\textbf{Claim 2:} $M_i\subseteq \calr_{i-d}M_d$ for any $i\geq d$.

We use induction on $i$ to prove the above claim. Our arguments proceed in different steps.

(i) First it is obvious that Claim 2 holds for $i=d$.

(ii) For any nonzero $m_1\in M_{d+1}$, thanks to the assumption (*), there exists $j\leq n$ such that
$D_1m_1=D_2m_1=\cdots =D_{j-1}m_1=0$, while $0\neq D_jm_1:=m_{1j}\in M_d$. Set $m_1^{\prime}=m_1-x_jm_{1j}\in M_{d+1}$. Then $D_1m_1^{\prime}=D_2m_1^{\prime}=\cdots =D_{j}m_1^{\prime}=0$. Proceed with this procedure yields that we can find $m_{1j},\cdots, m_{1n}\in M_d$ such that $D_i(m_1-\sum_{k=j}^nx_km_{1k})=0$ for any $1\leq i\leq n$. Hence, $m_1-\sum_{k=j}^nx_km_{1k}=0$ by the assumption (*), i.e., $m_1=\sum_{k=j}^nx_km_{1k}\in \mathscr{D}(M)_{d+1}$. Consequently, $M_{d+1}\subseteq\calr_1 M_{d}$.

(iii) Suppose $i\geq d+2$ and  $M_{i-1}\subseteq \calr_{i-d-1}M_d$. We need to show that $M_i\subseteq \calr_{i-d}M_d$.
For any nonzero $m_i\in M_i$, thanks to the assumption (*), there exists $j\leq n$ such that
$D_1m_i=D_2m_i=\cdots =D_{j-1}m_i=0$, while $0\neq D_jm_i:=m_{ij}\in M_{i-1}\subseteq \calr_{i-d-1}M_d$. Then
$D_1m_{ij}=D_2m_{ij}=\cdots =D_{j-1}m_{ij}=0$. Hence it follows from Lemma \ref{lem: basic prop} that we can write $m_{ij}$ as
$$m_{ij}=\sum_{k=1}^s\sum_{l=0}^{i-d-1}x_j^lf_{kl}m_i^{(k)}$$
where $$f_{kl}\in\bbf[x_{j+1},\cdots, x_n]\,\text{with}\,\deg(f_{kl})=i-d-1-l, 1\leq k\leq s, 0\leq l\leq i-d-1,$$
and $m_i^{(1)},\cdots, m_i^{(s)}\in M_d$ are linearly independent.
Set
$$m_{ij}^{\prime}:=\sum_{k=1}^s\sum_{l=0}^{i-d-1}\frac{1}{l+1}x_j^{l+1}f_{kl}m_i^{(k)}\in \calr_{i-d}M_d,$$
and $m_i^{\prime}:=m_i-m_{ij}^{\prime}\in M_i$. Then $D_1m_i^{\prime}=D_2m_i^{\prime}=\cdots =D_{j}m_i^{\prime}=0$.
Proceed with this procedure yields that we can find $m_{ij}^{\prime},\cdots, m_{in}^{\prime}\in \calr_{i-d}M_d$ such that $D_i(m_i-\sum_{k=j}^nm_{ik}^{\prime})=0$ for any $1\leq i\leq n$. Hence, $m_i-\sum_{k=j}^nm_{ik}^{\prime}=0$ by the assumption (*), i.e., $m_i=\sum_{k=j}^nm_{ik}^{\prime}\in \calr_{i-d}M_d$. Consequently, $M_i\subseteq \calr_{i-d}M_d$, as desired.

{ Summing up, we have $M_i=\mathscr{D}(M)_i$ for any $i\geq d$. This completes the proof of part (2).

For the part (3),  still keep  the assumption (*) in mind.} The arguments proceed in several steps.
			
(i) Note that $\ggg_0$ is a reductive Lie algebra, and that by the axioms of $\co$, $M$ is $\hhh$-semisimple, consequently $M_d$ is $\hhh$-semisimple (see \cite[Theorem 20.5.10]{TY}). Hence $M_d$ is completely reducible over $U(\ggg_0)$, i.e. $M_d=\bigoplus_{\text{finitely many }\lambda_k\in \Lambda^+} L^0(\lambda_k)$ (see for example, \cite[Theorem 20.5.10]{TY}).
				
(ii) Next we consider 
the Lie-Cartan submodule $M^{(i)}$ generated by $L^0(\lambda_i)$. 
				
%
%
%

(iii) According to Claim 2, $M^{(i)}=\calr L^0(\lambda_i)$, which coincides with the simple Lie-Cartan module $\calv(\lambda_i)$ by Lemma \ref{lem: basic prop} and Proposition \ref{prop: v lambda irr}. Hence Part (3) follows.

{We accomplish  the whole proof.}
%
}
\end{proof}


{{The submodule  $N$ in Lemma \ref{lem: radical}}} plays a role of the radical of $M$ in $\lc^+,$ and is denoted by $\text{Rad}_{\lc}(M)$.
%
 The following observation is clear.

\begin{lem}\label{lem: deltaR cv} For $\lambda\in \Lambda^+$.
$\Delta(\lambda)_\calr\slash \text{Rad}_{\lc}(\Delta(\lambda)_\calr)\cong \calv(\lambda)$.
\end{lem}
\begin{proof} Obviously, $\Delta(\lambda)_\calr$ is a Lie-Cartan module generated by its depth space $1\otimes L^0(\lambda)$. What remains is to apply Lemma \ref{lem: radical}(3).
\end{proof}

		\section{{{Classification of irreducible Lie-Cartan modules}}}
	{In this section, we classify all irreducible Lie-Cartan modules by some  arguments different from what  we made in Lemma \ref{lem: radical} (see Remark \ref{rem: diff way for clfy}). }	
	\subsection{} Let $\ggg=X(n)$ with $X\in\{W,S,H\}$ be a Lie algebra of vector fields in this section. We investigate simple  objects of $\lc$ and study the classification of isomorphism classes of irreducible Lie-Cartan modules. We first have the following observation.

\begin{lem} Any simple objects of $\lc$ belongs to $\lc^+$.
\end{lem}

 \begin{proof} It follows from Lemma \ref{lem: depth}.
  \end{proof}

\subsection{Property of \texorpdfstring{$\ggg$}{}-module homomorphism extensions}
\begin{lem}\label{lem: ext}	(\textsc{Extension lemma})
{Let $M$ and $N$ be two Lie-Cartan modules and $\phi: M\rightarrow N$  be a $\ggg$-module homomorphism. Suppose that $M'$ is a $\ggg$-submodule  of $M$, with $M'=\sum_{i\geq d}M'_i$ satisfying $M'=U(\ggg)M'_d$ for $d=d(M')$. If $\phi$   satisfies the following condition

					\begin{align}\label{eq: condition}
						\sum_{i=1}^{m} u_if_iv_i=0 \text{  implies } \sum_{i=1}^{m} u_if_i\phi(v_i)=0
					 \text{ for }u_i\in U(\ggg), f_i\in\calr, v_i\in M'_d,
					\end{align}
	then $\phi|_{M'}$ can be extended to a Lie-Cartan module homomorphism $\Phi$ from $\widetilde{M'}$ to $N$ where  $\widetilde{M'}$ is Lie-Cartan submodule of $M$ generated by $M'$.}
	\end{lem}
			
	\begin{proof} The assumption on $M'$ along with  the axioms of Lie-Cartan modules yield
				\begin{align}\label{eq: tilde M prime}
					\widetilde{M'}&=\calr M'\cr
					&=U(\ggg_{>0})\calr M'_d.
				\end{align}
	Note that the depth of $\widetilde{M'}$ is also equal to $d$, and that
	its depth space is exactly $M'_d$. By Lemma \ref{lem: basic prop},
	we see that $\calr M'_d$ is a free $\calr$-submodule in $\widetilde{M'}$. By the same reason, in the Lie-Cartan submodule generated by $\phi(M')$ we have that $\calr\phi(M'_d)$  { is zero or also a free submodule over $\calr$. We can extend $\phi$ to the desired Lie-Cartan module homomorphism $\Phi$ via:
					$$\Phi(X  fv)=X   f \phi(v)\,\text{ for }X \in\ggg, f\in\calr \text{ and }v\in M'_d,$$
					which is well-defined by (\ref{eq: condition}).}	
			\end{proof}

		{{	\subsection{Irreducible Lie-Cartan modules}
				\begin{lem}\label{lem: M_0 irr} Let $\ggg=X(n)$, $X\in\{W,S,H\}$.
					Suppose that $M=\sum_{i\geq0}M_i$ is an irreducible Lie-Cartan module. Then as a Lie-Cartan module, $M_0$ is a generated space. Furthermore, $M_0$ is an irreducible $\ggg_0$-module.
				\end{lem}
				
				\begin{proof} An irreducible Lie Cartan module can be generated by any nonzero element. In particular, it is generated by $M_0$. Keeping the axiom (LC-1) in mind, we can write $M=\theta(\calr) U(\ggg)M_0$.  Note that $M$ is by definition an object in $\co$, then $M_0$ is a finite dimensional $\ggg_0$-module with diagonal $\hhh$-action. 					Hence, by the classical Lie theory it is easily known that $M_0$ is completely reducible over $\ggg_0$, that is, as $\ggg_0$-modules, $M_0\cong\bigoplus_{i=1}^r L^0(\lambda_i)$ for some $\lambda_i\in \Lambda^+$.
					
					We continue to show that $r=1$. If otherwise, suppose $r>1$. By the above arguments,  for simplicity we write $M_0=\bigoplus_{i=1}^r L^0(\lambda_i)$.  Consider $M^{(i)}$, which  is the Lie-Cartan submodule generated by $L^0(\lambda_i)$ for $1\leq i\leq r$.  We will show that all $M^{(i)}$ are different submodules, which consequently contradicts the irreducible assumption of $M$.
					
					Actually, $M^{(i)}=\theta(\calr)U(\ggg_{>0})L^0(\lambda_i)$. By definition (the axiom (LC-2)), $\theta$ preserves the degree of $\calr$. Hence the depth space of $M^{(i)}$ is just $L^0(\lambda_i)$.  Hence  all $M^{(i)}$ are different.
					
					By summation, the proof is completed.
				\end{proof}

				\begin{prop}\label{prop: irr LC} Let  $\ggg=X(n)$, $X\in\{W,S,H\}$. Any irreducible Lie-Cartan module is isomorphic to $\cv(\lambda)$ for some $\lambda\in {\Lambda^+}$.
				\end{prop}
				
				\begin{proof} Suppose that $M$ is an arbitrarily given irreducible Lie-Cartan module of depth $d=d(M)$, which can be expressed as $M=\sum_{i\geq d}M_i$. By irreducibility and the definition of Lie-Cartan modules,  $M=U(\ggg)\calr M_d$. According to Lemma \ref{lem: M_0 irr}, $M_d\cong L^0(\lambda)$ for some $\lambda\in \Lambda^+$. Note that $M$ is still an object of $\co$. As a $U(\ggg)$-module, $M$ still admits the depth space $M_d\cong L^0(\lambda)$. Hence by the same arguments as in Lemma \ref{lem: radical}, $M$ has an irreducible quotient $L(\lambda)$ in $\co$, which yields that there is a $U(\ggg)$-homomorphism $\phi: M\rightarrow L(\lambda)$.  {In particular, $\phi|_{M_d}: M_d\rightarrow L^0(\lambda)$ is a nonzero homomorphism of irreducible $\ggg_0$-modules, {{so that  $\phi|_{M_d}$ is a constant map by a nonzero scalar.}}} 
					Note that $L(\lambda)$ is contained in a Lie-Cartan module $\cv(\lambda)$ (see Example \ref{ex: Lie-Cartan containing}(2)), and  that $\cv(\lambda)$ is already known to be an irreducible Lie-Cartan module (see Proposition \ref{prop: v lambda irr}), and $\phi$  satisfies   the condition (\ref{eq: condition}) {{ because $\phi|_{M_d}$ is already known as a scalar map}}. By Extension Lemma \ref{lem: ext}, there is a Lie-Cartan module homomorphism from $M$ to $\cv(\lambda)$. The irreducibility of $M$ and $\cv(\lambda)$ ensures that $\Phi$ is an isomorphism of Lie-Cartan modules. The proof is completed.
				\end{proof}

{				
				
\begin{rem}\label{rem: diff way for clfy} The above proposition can be also   proved by application of Lemma \ref{lem: radical}.
\end{rem}				
}				
				
				\subsection{Classification of irreducible Lie-Cartan modules}
				
				We are now in the position to present the following classification theorem.

				\begin{thm}\label{thm: irre-mod}
					The set $\{{}^{d}\hspace{-2pt}\cv(\lambda)=\calr\otimes {}^dL^0(\lambda)\mid \lambda\in {\Lambda^+}, d\in \bbz\}$ exhausts all non-isomorphic irreducible  Lie-Cartan modules.
				\end{thm}
				\begin{proof}
					Note that ${}^d\cv(\lambda)_d=L^0(\lambda)$ for $\lambda\in {\Lambda^+}$. Hence, ${}^d\cv(\lambda)\cong {}^l\cv(\mu)$ as Lie-Cartan modules implies that  $d=l$, and $L^0(\lambda)\cong L^0(\mu)$ as $\ggg_0$-modules, yielding that $\lambda=\mu$. Consequently, the desired result follows directly from  Proposition \ref{prop: irr LC}.
				\end{proof}
		}}

\section{Homological properties of Lie-Cartan modules}

		{
			\subsection{Hom spaces}
			For preciseness and simplicity we denote $\cv(\lambda)$ by ${}^{d}\hspace{-2pt}\cv(\lambda)$ whenever $\cv(\lambda)$ ($\lambda\in \Lambda^+$) admits the depth $d$. Similarly, we denote $\Delta(\lambda)_\calr$ by ${}^{d}\hspace{-2pt}\Delta(\lambda)_\calr$ the standard module in $\lc$ with depth $d$.
			
			\begin{prop} The following statements hold.
				\begin{itemize}
					\item[(1)]
					\begin{equation*}
						\dim\Hom_\lc({}^{d_1}\hspace{-2pt}\cv(\lambda),{}^{d_2}\hspace{-2pt}\cv(\mu))=\begin{cases}
							1, &\text{ if  } d_1= d_2 \text{ and }\lambda= \mu;\cr
							0, &\text{ otherwise}.
						\end{cases}
					\end{equation*}
					
					\item[(2)]{
					\begin{equation*}
						\dim\Hom_\lc({}^{d_1}\hspace{-2pt}\Delta(\lambda)_\calr,
						{}^{d_2}\hspace{-2pt}\Delta(\mu)_\calr)=\begin{cases}
							m(\lambda), &\text{ if  } d_1\geq d_2;\cr
							0, &\text{ otherwise},
						\end{cases}
					\end{equation*}
					where
\[
\begin{array}{ccl}
	\cala_{d_1, d_2}&:=&\text{Ann}_{\ggg_{-1}}({}^{d_2}\hspace{-2pt}\Delta(\mu)_\calr)_{d_1}:=\{ v\in({}^{d_2}\hspace{-2pt}\Delta(\mu)_\calr)_{d_1}\mid \ggg_{-1}.v=0\}\\
	&~&~\\
	&=&\bigoplus_{\gamma\in\Lambda^+}m(\gamma){\hspace{2pt}}L^0(\gamma)
\end{array}
\]
is the semisimple decomposition of $\cala_{d_1, d_2}$ over $\ggg_0$.
					
					\item[(3)]
					 In particular, $\Hom_\lc({}^{d}\hspace{-2pt}\Delta(\lambda)_\calr,
						{}^{d}\hspace{-2pt}\Delta(\lambda)_\calr)\cong\bbf$.
						\item[(4)]
						\begin{align*}
							\dim\Hom_{\lc}({}^{d}\hspace{-2pt}\Delta(\lambda)_\calr, {}^{d'}\hspace{-2pt}\cv(\mu))
							=\begin{cases}1,   &\text {if }\lambda=\mu \text{ and }d=d';\cr
								0, &\text{otherwise}.
							\end{cases}
						\end{align*}
						
					}
				\end{itemize}
			\end{prop}
		}
		
		{
			\begin{proof} Part (1) is due to
 the classification theorem of irreducible Lie-Cartan modules (see Theorem \ref{thm: irre-mod}). Part (3) is a corollary to (2). The proof of (4) is similar to (2).
				
				We only need to prove (2). Note that the Lie-Cartan module ${}^{d_1}\hspace{-2pt}\Delta(\lambda)_\calr$ is generated by its depth space $L^0(\lambda)$. When $d_1<d_2$, then for any homomorphism $\phi:{}^{d_1}\hspace{-2pt}\Delta(\lambda)_\calr\rightarrow {}^{d_2}\hspace{-2pt}\Delta(\mu)_\calr$, the image of $\phi$ must be zero. Hence $\phi$ has to be zero.
				
				Suppose $d_1\geq d_2$. By the same reason as above, to describe such a homomorphism $\phi$ we only need to determine its image of the depth space ${}^{d_1}\hspace{-2pt}L^0(\lambda)$ in ${}^{d_2}\hspace{-2pt}\Delta(\mu)_\calr$. The image must be a subspace of  $\cala_{d_1,d_2}$, which is a $\ggg_0$-module, because the depth space ${}^{d_1}\hspace{-2pt}L^0(\lambda)$ is annihilated by $\ggg_{-1}$, and any homomorphism preserves the gradings.
				This yields that $\dim\Hom_\lc({}^{d_1}\hspace{-2pt}\Delta(\lambda)_\calr,
				{}^{d_2}\hspace{-2pt}\Delta(\mu)_\calr)=
				m(\lambda)$.
			\end{proof}

		}
		
		\subsection{} Now we introduce cohomology theory for the category $\lc$.
		In the following, we look at the extensions between some modules, especially between irreducible modules $\cv(\lambda)$ for $\lambda\in\Lambda^+$.  %
								
		\begin{prop}\label{prop: compute ext} Let $\lambda,\mu\in \Lambda^+$.  Suppose that $\cv(\lambda)$ and $\cv(\mu)$ admit the depth $d$.
The following statements hold.
			\begin{itemize}
				\item[(1)] $\cv(\lambda)$ admits a projective cover ${}^d\hspace{-2pt}\Delta(\lambda)_\calr$ in $\lc_{\geq d}$.
				\item[(2)] $\cv(\lambda)$ is an injective object in $\lc_{ \geq d}$.
				
				\item[(3)] $\Ext_{\lc_{\geq d}}^1({}^d\hspace{-2pt}\Delta(\lambda)_\calr, \cv(\mu))=0$ for any $\lambda,\mu$.
			\end{itemize}
		\end{prop}
		
		\begin{proof}
			{
			(1)	Let us consider the following diagram in the category $\lc_{\geq d}$:				
				\begin{equation*}
					\begin{array}{c}
						\xymatrix{
							& \Delta(\lambda)_\calr\ar[d]^{\psi}&\\
							M \ar[r]_{\varphi} & N\ar[r]&0 .}
					\end{array}
				\end{equation*}
				where $\varphi$ ia an epimorphism. Observe that $ \psi_0: \Delta(\lambda) \to  \Delta(\lambda)_\calr$, $v \mapsto 1 \otimes v$, $v \in \Delta(\lambda)$ define a $\ggg$ module homomorphism. Hence $\psi \circ \psi_0$ is a $\ggg$-module homomorphism. Note that  $\Delta(\lambda)$ is a projective object in the category $\mathcal{O}_{\geq d}$ (see \cite[Lemma 4.1]{DSY}). Therefore we can find a $\ggg$-module homomorphism $\alpha:\Delta(\lambda) \rightarrow M$ such that the following diagram commute:			
				
				\begin{equation*}
					\begin{array}{c}
						\xymatrix{
							& \Delta(\lambda)\ar[dl]_{\alpha}\ar[d]^{\psi\circ\psi_0}& \\
							M \ar[r]_{\varphi} & N\ar[r]&0.}
					\end{array}
				\end{equation*}

				Now define a map $ \widetilde \psi:\Delta(\lambda)_\calr \to M$ by:
				$$\widetilde \psi(r \otimes v)=r.\alpha(v)  .$$
				It is easy to see that this is the desired map in the category $\lc_{\geq d}$ to make the following diagram commutes.
\begin{equation*}
					\begin{array}{c}
						\xymatrix{
							& \Delta(\lambda)_{\calr}\ar[dl]_{\tilde{\psi}}\ar[d]^{\psi} &\\
							M \ar[r]_{\varphi} & N\ar[r]&0.}
					\end{array}
				\end{equation*}
This proves that $\Delta(\lambda)_\calr$ is a projective object in the category
				$\lc_{\geq d}$.

{Suppose the depths of both $\Delta(\lambda)_\calr$ and $\cv(\lambda)$ are the same $d$.  By Lemma \ref{lem: deltaR cv}, there is a surjective Lie-Cartan module homomorphism $$\pi:\Delta(\lambda)_\calr\rightarrow \Delta(\lambda)_\calr\slash \text{Rad}_{\lc}(\Delta(\lambda)_\calr)\cong \cv(\lambda).$$
  Any proper submodule of $\Delta(\lambda)_\calr$ admits depth bigger than $d$. So there is no nonzero Lie-Cartan module homomorphism from any given  proper submodule of $\Delta(\lambda)_\calr$ to $\cv(\lambda)$, which yields that  $\pi$ is an essential map (and $\Delta(\lambda)_\calr$ is indecomposable). This completes the proof of (1).}
				
To prove (2) assume that $i: M\longrightarrow N$ is a monomorphism of two modules $M$ and $N$ in $\lc_{ \geq d}$, and there is a non-zero homomorphism $\rho: M \rightarrow \cv(\lambda)$ in the category $\lc_{\geq d}$. Then irreducibility of $ \cv(\lambda)$ implies that $\rho$ is surjective, and factors through $\text{Rad}_{\lc}(M)$, i.e, it induces a non-zero homomorphism  $\bar{\rho}: M/\text{Rad}_{\lc}(M) \rightarrow \cv(\lambda)$ such that $\rho=\bar{\rho}\circ\pi_M$ where $\pi_M: M\longrightarrow M/\text{Rad}_{\lc}(M)$ is the canonical morphism, and $\cv(\lambda)$ is also a direct summand of the semisimple module $M/\text{Rad}_{\lc}(M)$. Since $i^{-1}(\text{Rad}_{\lc}(N))=\text{Rad}_{\lc}(M)$, the monomorphism $i$ induces the corresponding monomorphism $\bar{i}: M/\text{Rad}_{\lc}(M)\longrightarrow N/\text{Rad}_{\lc}(N)$. Note that $M/\text{Rad}_{\lc}(M)$ and $N/\text{Rad}_{\lc}(N)$ are semisimple, we have a surjective morphism $\zeta: N/\text{Rad}_{\lc}(N)\longrightarrow \cv(\lambda)$ such that the following diagram commutes.
\begin{equation*}
		\begin{array}{c}
						\xymatrix{
                          0\ar[r]&M\ar[r]^i\ar[d]_{\pi_M}&N\ar[d]_{\pi_N}\\
                          &M/\text{Rad}_{\lc}(M)\ar[r]^{\bar{i}}\ar[d]_{\bar{\rho}}&N/\text{Rad}_{\lc}(N)\ar[dl]^{\zeta}\\
							&  \cv(\lambda) &.}
					\end{array}
\end{equation*}
Then $\varrho:=\zeta\circ\pi_N:N\longrightarrow\cv(\lambda)$ is the desired homomorphism such that $\rho=\varrho\circ i$.
We complete the proof of (2).

				(3) immediately follows from (1).				
			}	
		\end{proof}

		{
			We have the following corollary to Proposition \ref{prop: compute ext}.
			
			\begin{cor}\label{cor: ext exp} In the category $\lc_{\geq d}$, the following statements hold.
				\begin{itemize}
					\item[(1)] $\text{Ext}^i({}^{d'}\hspace{-2pt}\cv(\lambda),{}^d\hspace{-2pt}\cv(\mu))=0$ for $i>0$ and $\lambda,\mu\in \Lambda^+$.
					
					\item[(2)]
					$\text{Ext}^i({}^d\hspace{-2pt}\Delta(\lambda)_\calr,M)=0$ for  $i>0$ and $M\in \lc_{\geq d}$.

\item[(3)] { We have
\begin{align*}
\Ext^1_{\lc_{\geq d}}({}^{d}\hspace{-2pt}\calv(\lambda), {}^{d'}\hspace{-2pt}\calv(\mu))
=\begin{cases}
0,    &\text{ if }d=d';\cr
\cong \Hom_{\lc_{\geq d}}(\text{Rad}_{\lc}({}^{d}\hspace{-2pt}\Delta(\lambda)_\calr),{}^{d^{\prime}}\calv(\mu)),     &\text{ if } d<d'.
\end{cases}
\end{align*}

}
				\end{itemize}
			\end{cor}
		}
{		
	\begin{proof}
Since ${}^d\hspace{-2pt}\cv(\mu)$ is injective in $\lc_{ \geq d}$ by Proposition \ref{prop: compute ext}(2), Part (1)  follows. While  ${}^d\hspace{-2pt}\Delta(\lambda)_\calr$ is projective by Proposition \ref{prop: compute ext}(1), Part (2) holds.

{
As to Part (3), consider the short exact sequence

{
$$ 0\rightarrow \text{Rad}_{\lc}({}^{d}\hspace{-2pt}\Delta(\lambda)_\calr)\rightarrow   \,^d\hspace{-2pt}\Delta(\lambda)_\calr\rightarrow\,^d\hspace{-1pt}\calv(\lambda)\rightarrow 0.$$
}
Applying the left exact functor $\Hom_{\lc_{\geq d}}(-, {}^{d'}\hspace{-2pt}\calv(\mu))$ to the above short exact sequence, we get a long exact sequence, from which Part (3) follows  by Proposition \ref{prop: compute ext}(3).
}
\end{proof}	
}

		{
			
			\section{Cohomology for  universal Lie-Cartan modules}
	{In this section, we focus our concerns on the whole category $\lc$ and get some basic understandings.  Especially, we develop the so-called $\ulc$-cohomology theory and present some interesting results.}

 {
\subsection{Questions}	Note that $\calr$ itself is a natural Lie-Cartan module. Actually, $\calr$ is isomorphic to $\calv(0)$, as a Lie-Cartan module. Set $\mathcal{E}:=\text{Ext}_{\lc}^\bullet(\calr,\calr)$.
The following questions naturally arise.

			\begin{question}\label{ques: ext}
				\begin{itemize}
					\item[(1)] What does $\mathcal{E}$ look like? Is it  a finitely-generated commutative algebra?

					\item[(2)] Is $\mathcal{M}:=\text{Ext}^\bullet_{\lc}(M,M)$ a finitely-generated $\mathcal{E}$-module for any $M\in \lc^+$? In particular, what about $\Ext^\bullet(\calv(\lambda),\calv(\lambda))$?
					
					\item[(3)] { When is the extension nonzero in Corollary \ref{cor: ext exp}(3)?}
				\end{itemize}
			\end{question}
		
These questions are not trivial.  Currently, we can give some partial answer in the so-called universal Lie-Cartan module category (see Definition \ref{defn: univ LC}).

\subsection{Universal Lie-Cartan modules}
\begin{defn}\label{defn: univ LC} Let $\ggg=W(n)$, $S(n)$ or $H(n)$.  The universal Lie-Cartan module category $\ulc$ is a subcategory of $U(\ggg)$-modules whose objects are also $\calr$-modules satisfying  the following axioms
\begin{itemize}

\item[(1)] Any object $M$ is endowed with $\ggg$-action $\rho$, $\calr$-action $\theta$ satisfying the equation
$$ [\rho(X),\theta(f)]=\theta(X(f))
$$
for $f\in\calr$ and $X\in\ggg$.

\item[(2)] {Morphisms  between objects are homomorphisms  of $\ggg$-modules and of $\calr$-modules.}
\end{itemize}
\end{defn} 	

Clearly, $\lc$ is a full subcategory of $\ulc$.

\subsubsection{ { The naturalized Lie algebra and naturalized  associative algebra associated with a derivation Lie algebra of  a commutative algebra}}
		Let $\scrl$ be a Lie algebra over $\bbf$, and $\cala$ be a commutative and associative $\bbf$-algebra with $\scrl$-derivation action.  		{We define a new Lie algebra $\cala\natural\scrl$ whose underling space is $\cala\oplus\scrl$ with Lie products as below.
		\begin{align}\label{eq: sharp prod}
			&[f,g]=0 &\text{for }f,g\in \cala,\cr
			&[X, g]=X(g)=-[g,X] &\text{for } X\in\scrl, g\in\cala.
		\end{align}
		We call $\cala\natural\scrl$ a naturalized Lie algebra associated with $\cala$ and $\scrl$.}

{ Next we introduce  an algebra  $\cala\natural U(\scrl)$.}

{
 \begin{defn}\label{def: nat alg} The algebra $\cala\natural U(\scrl)$ is   $\cala\otimes U(\scrl)$ as a vector space, endowed with  multiplication structure defined via the following axiom.
 \begin{itemize}


\item[(N1)] $(a_1\otimes u_1 )(a_2\otimes u_2)=\sum_{i} a_1a_{2i}\otimes u_{i1} u_2$ for $a_1, a_2\in\cala$ and $u_1, u_2\in U(\scrl)$. Here    $u_1a_2=\sum_{i}a_{2i}u_{i1}$ in $U(\cala\natural\scrl)$.
\end{itemize}
{The forthcoming Lemma \ref{lem: associative} ensures that the above axiom (N1) make sense.} We call $\cala\natural U(\scrl)$ the naturalized  algebra associated with $\cala$ and $\scrl$.
\end{defn}

\begin{rem}
The naturalized algebra $\cala\natural U(\scrl)$ coincides with the smash product $\cala\# U(\scrl)$ defined in \cite{BIN} (see also \cite{Mon} for the general case) when $\cala=\calr$ and $\scrl=W(n)$.
\end{rem}

\begin{lem}\label{lem: associative}
{The multiplication defined in (N1) is  associative.}
\end{lem}
\begin{proof}
Take any $a,b,c\in\cala, u_1,u_2, u_3\in U(\scrl)$, assume that $u_1b=\sum_i b_iu_{i1}$, $u_2c=\sum_j c_ju_{j2}$, and $u_{i1}c_j=\sum_k c_{jk}u_{ki1}$ in $U(\cala\natural \scrl)$. Then
\begin{align*}
&\big((a\otimes u_1)(b\otimes u_2)\big)(c\otimes u_3)\cr
=&\sum_i(ab_i\otimes u_{i1}u_2)(c\otimes u_3)\cr
=&\sum_{i,j,k}ab_ic_{jk}\otimes u_{ki1}u_{j2}u_3.
\end{align*}
On the other hand,
\begin{align*}
&(a\otimes u_1)\big((b\otimes u_2)(c\otimes u_3)\big)\cr
=&(a\otimes u_1)\Big(\sum_jbc_j\otimes u_{j2}u_3\Big)\cr
=&\sum_{i,j,k}ab_ic_{jk}\otimes u_{ki1}u_{j2}u_3.
\end{align*}
Hence, $\big((a\otimes u_1)(b\otimes u_2)\big)(c\otimes u_3)=(a\otimes u_1)\big((b\otimes u_2)(c\otimes u_3)\big)$, as desired.
\end{proof}
}

{
Roughly speaking,  $\cala\natural U(\scrl)$ is an algebra defined on $\cala\otimes U(\scrl)$ satisfying  (\ref{eq: sharp prod}).  More precisely, we have the following
\begin{itemize}
\item[(N2)] $U(\scrl)\hookrightarrow \cala\natural U(\scrl)$, $u\mapsto 1\otimes u$ is an imbedding of associative algebras.

\item[(N3)] $\cala\hookrightarrow \cala\natural U(\scrl)$, $a\mapsto a\otimes 1$ is an imbedding of associative algebras.
\item[(N4)] $(a\otimes 1)(1\otimes X)=a\otimes X$ for $a\in\cala$ and $X\in\ggg$.
    \item[(N5)] $(1\otimes X) (a\otimes 1)=a\otimes X-X(a)\otimes 1$
for all $a\in \cala$ and $X\in\ggg$.

\end{itemize}

}

\subsubsection{An equivalent definition of universal Lie-Cartan modules}
{

\begin{lem}\label{lem: tensor r ulc} The following statements hold.
\begin{itemize}
\item[(1)]  { $\ulc$ is equivalent to the category of $\calr\natural U(\ggg)$-modules.}
\item[(2)]  For any $\ggg$-module $M$, $\calr\otimes_\bbf M$ becomes a universal Lie-Cartan module with natural $\calr$-, $\ggg$-action in the same sense as Lemma \ref{lem: M_R}.
    \end{itemize}
\end{lem}

\begin{proof} (1) It directly follows from the definition.

(2) By the same arguments as in the proof of Lemma \ref{lem: M_R}, it can be proved.
\end{proof}

}

\subsection{\texorpdfstring{$\ulc$}{}-cohomology}

\subsubsection{}  Let us first introduce the unity map $\omega: \calr\longrightarrow\bbf$, which is by definition an algebra homomorphism with $\ker\omega=\sum_{i=1}^n\calr x_i$.
\subsubsection{}
{ Recall that in the ordinary cohomology theory for Lie algebras, one can present the cohomology through the Chevalley-Eilenberg complex (see \cite{Fuks}, \cite{Kn}, \cite[\S7.7]{Wei}, {\sl{etc.}}). Denote by $H^q(\ggg)$ and $H^q(\ggg,M)$  respectively the $q$th cohomology of $\ggg$ and  the $q$th cohomology of $\ggg$ with coefficient in the $\ggg$-module $M$.
 We first have the following key result.}
				
\begin{prop}\label{lem: basic coho}
Keep the notations as before.  In particular, let $\ggg=W(n)$, $S(n)$ or $H(n)$. The following statements hold. 							
\begin{itemize}								
\item[(1)]  All $q$th wedge products $\bigwedge^q\ggg$ of $\ggg$ fall in $\co^+$ for $q\in\mathbb{N}$. Correspondingly, $\calr\otimes_\bbf \bigwedge^q\ggg$ lies in $\lc$.								
								
\item[(2)]  Set $C_{\ulc}^q:=\Hom_\calr(\calr\otimes_\bbf\bigwedge^q\ggg, \calr)$ for $q\in\mathbb{N}$. Then $C_\ulc^q$ is an $\calr$-module by the unity  action $\Omega$, i.e.
    \begin{align}\label{eq: unity}
    (\Omega(f)\phi)(y)=\omega(f)\phi(y)
    \end{align}
    for $\phi\in C_{\ulc}^q$,  $f\in\calr$ and $y=r\otimes w\in\calr\otimes_\bbf\bigwedge^q\ggg$, along with a $\ggg$-module structure. {Such a module will be called a $(\ggg,\calr)$-module on which we don't consider any Lie-Cartan module structure. }
								
\item[(3)] The differentials $\{ \sfd_q\mid q=0,1\ldots,\}$						with
$C_\ulc^q{\overset{\sfd_q}\longrightarrow} C_\ulc^{q+1}$
								defined via
{
\begin{align*}
&\sfd_q \phi(r\otimes (g_1\wedge\dots\wedge g_{q+1}))\cr
									=&\displaystyle{\sum_{1\leq s<t\leq q+1}}(-1)^{s+t}\phi(r\otimes([g_s,g_t]\wedge g_1\wedge\dots\wedge\hat g_s \wedge\dots \wedge\hat g_t \wedge \dots \wedge g_{q+1})) \cr
									&+\displaystyle{\sum_{1\leq s\leq q+1}}(-1)^{s+1}(rg_s\phi(1\otimes (g_1\wedge\dots\wedge\hat g_s\wedge\dots\wedge g_{q+1})))		\end{align*}
								for $g_1,\ldots,g_{q+1}\in\ggg$ and $r\in\calr$,
								define the cochain complex $$\bfC=(C_{\ulc}^q, \sfd:=\sfd_q)$$
in the $(\ggg,\calr)$-module category.}
								
								\item[(4)] {The cohomology space $H^q(\bfC)$ is isomorphic to $H^q(\ggg,\calr)$.}
							\end{itemize}			
						\end{prop}
					
						\begin{proof}
					To prove (1) note that $\ggg=\bigoplus\limits_{i=-1}^{\infty}\ggg_{i}$ with all $\ggg_i$ being finite dimensional. Therefore one can write  $\bigwedge^q\ggg=\bigoplus\limits_{ l\geq -q}^{\infty}(\bigwedge^q\ggg)_{l}$, where $(\bigwedge^q\ggg)_{l}=\bigoplus\limits_{ i_1+\dots +i_q = l} \ggg_{i_1}\wedge \dots \wedge\ggg_{i_q}$ with all $(\bigwedge^q\ggg)_l$ being finite dimensional.	Again $\ggg$ is a locally finite $\sfp$-module, which implies that $\bigwedge^q\ggg$ is a locally finite $\sfp$-module. There is nothing to prove for the fact that $(\bigwedge^q\ggg)$ is a weight module.  Now by Lemma \ref{lem: M_R}, $\calr \otimes (\bigwedge^q\ggg)$ is a Lie-Cartan module. This completes the proof of (1).

For part (2), $C_{\ulc}^q$ is clearly an $\calr$-module. We further check its $\ggg$-module structure.
We usually introduce  a $\ggg$-action $\rho$ on $C_{\ulc}^q$ as follows.
\begin{align}\label{eq: def g action}
(\rho(X)\phi)(y):=X(\phi(y))-\phi(X.y)
 =X(\phi(y))-\phi(X(r)\otimes w)-\phi(r\otimes X.w),
 \end{align}
where $X\in\ggg$, $\phi\in C_{\ulc}^q$ and $y=r\otimes w\in\calr\otimes_\bbf\bigwedge^q\ggg.$
Then
 it is readily proved that $C_{\ulc}^q$ is a $\ggg$-module under the action $\rho$. 

 Now we check  (3).  We need to verify that
  \begin{itemize}
  \item[(3-1)] $\sfd_q$ is a well defined mapping, i.e. $\sfd_q\phi$ is an $\calr$-module homomorphism for any $\phi\in C_{\ulc}^q$;
  \item[(3-2)]  $\sfd_q$ is a homomorphism of both $\ggg$-modules and $\calr$-modules;
  \item[(3-3)] $\sfd_q$ is a cochain homomorphism, i.e. $\sfd_{q+1}\circ\sfd_q=0$.
  \end{itemize}
  For (3-1), note that $\phi$ is an $\calr$-module homomorphism. So $f\phi(y)=\phi(fy)$ for $f\in\calr$ and $y=r\otimes w\in \calr\otimes_\bbf \bigwedge^q\ggg$.
  Then it is a routine to check that $\sfd_q \phi$ is an $\calr$-module homomorphism by the definition of $\sfd_q.$
  For (3-2), we have
 \begin{align*}
&\Omega(f)(\sfd_q \phi)\cr
=&\omega(f)\sfd_q(\phi)\cr
=&\sfd_q (\omega(f)\phi)\cr
=&\sfd_q(\Omega(f)\phi).
			\end{align*}
So we have already proved that $\sfd_q$ is an $\calr$-module homomorphism.

Next, we show that $\sfd_q$ is a $\ggg$-module homomorphism. Note that $$C_{\ulc}^q=\Hom_\calr(\calr\otimes \bigwedge\nolimits^q\ggg,\calr)\cong \Hom_\bbf(\bigwedge\nolimits^q\ggg,\calr).$$
So we first claim  that $\sfd'_q:={\sfd_q}|_{\Hom_{\bbf}(\bbf\otimes \bigwedge^q\ggg,\calr)}$ is a $\ggg$-module homomorphism. By definition, $\sfd'_q$ sends  $\phi$ to $\sfd'_q(\phi)$ via
\begin{align*}
&\sfd'_q(\phi)(1\otimes g_1\wedge\dots\wedge g_{q+1})\cr
	=&\displaystyle{\sum_{1\leq s<t\leq q+1}}(-1)^{s+t}\phi(1\otimes [g_s,g_t]\wedge g_1\wedge\dots\wedge\hat g_s \wedge\dots \wedge\hat g_t \wedge \dots \wedge g_{q+1}) \cr
	&+\displaystyle{\sum_{1\leq s\leq q+1}}(-1)^{s+1}g_s.\phi(1\otimes g_1\wedge\dots\wedge\hat g_s\wedge\dots\wedge g_{q+1})
\end{align*}
This claim is due to the fact that with $\sfd_q'$ we return to the situation of  the classical Chevalley-Eilenberg complex. Such $\sfd'_q$ is surely a $\ggg$-module homomorphism. Now we already know that both  $\sfd_q$ and $\phi$ are $\calr$-module homomorphisms. Write $w:=g_1\wedge\dots\wedge g_{q+1}$.
Then for any $X\in\ggg$, we have
\begin{align*}
& (\rho(X)(\sfd_q(\phi)))(r\otimes w)\cr
=&X ((\sfd_q\phi)(r\otimes w))-(\sfd_q\phi)(X(r\otimes w))\cr
=&(X(r)\sfd_q'\phi(1\otimes w)+r X(\sfd_q'\phi(1\otimes w))
-(X(r)\sfd_q'(\phi)(1\otimes w)+r\sfd_q'(\phi)(1\otimes X.w))\cr
=&r X(\sfd_q'\phi(1\otimes w))
-r\sfd_q'(\phi)(1\otimes X.w))\cr
=&r (\rho(X)(\sfd_q'\phi))(1\otimes w)\cr
=&r (\sfd_q'(\rho(X)\phi))(1\otimes w)\cr
=&\sfd_q(\rho(X)\phi)(r\otimes w)).
\end{align*}
The above last equation is because $\sfd_q(\rho(X)\phi)$ is an $\calr$-homomorphism by (3-1).
Hence $\rho(X)\sfd_q(\phi)=\sfd_q(\rho(X)\phi)$. So $\sfd_q$ is indeed a $\ggg$-module homomorphism.

 Now we check (3-3). For $y=r\otimes w\in \calr\otimes \bigwedge^{q+2}\ggg,\phi\in C_{\ulc}^q$, by Classical Chevalley-Eilenberg complex we have $(\sfd_{q+1}'\circ\sfd_{q}'(\phi))(1\otimes w)=0$. Hence we finally have
 $$ (\sfd_{q+1}\circ\sfd_{q}(\phi))(y)=r(\sfd_{q+1}'\circ\sfd_{q}'(\phi))(1\otimes w)=0. $$
 This completes the proof of Part (3).

 For Part (4) it follows from the fact that
 $\Hom_\calr (\calr\otimes_\bbf \bigwedge^q\ggg, \calr) \cong \Hom_\bbf (\bigwedge^q\ggg , \calr)$.

 Summing up, we accomplish the proof.	
\end{proof}
						
	{
\subsubsection{ \texorpdfstring{$\ulc$}{} -cohomology}
We are going to  introduce the cohomology of $\ulc$, which will be simply called the $\ulc$-cohomology.
By Lemma \ref{lem: tensor r ulc}, a universal Lie-Cartan module is actually a $\calr\natural U(\ggg)$-module.

\begin{lem} For any given $M\in\ulc$, set $\Gamma(M)=\Hom_{\ulc}(\calr,M)$. Then $\Gamma$ defines a left exact functor from $\ulc$ to the category of $(\ggg,\calr)$-modules.
\end{lem}

\begin{proof} Let  $(\rho_0,\theta_0)$ be the structure mapping pair on the Lie-Cartan module $M$.  We  make arguments in steps.

(1) $\Gamma(M)$ is endowed with trivial $\ggg$-module structure. Actually, by definition we have $\rho_0(X)(\phi(r))=\phi(X(r))$, equivalently, $\rho(X)\phi=0$ where $\phi\in\Gamma(M)$, $X\in\ggg, r\in\calr$, $\rho$ is usually defined as in (\ref{eq: def g action}).

(2) Consider the action $\theta$ of $\calr$ on $\Gamma(M)$ by the unity action defined in (\ref{eq: unity}).
 It is obvious that $\Omega(f)\phi$ is still in $\Gamma(M)$, since $\Omega(f)\phi=\omega(f)\phi$.
The associative property is easily  confirmed. Hence $\Gamma(M)$ is an $\calr$-module.

The left exactness comes from the general property of the hom functor.
The proof is completed.
\end{proof}

We will denote by $(\ggg,\calr)\hmod$ the category of $(\ggg,\calr)$-modules.

\begin{defn}\label{def: ulc coh} We define the $q$th $\ulc$-cohomology  with coefficient in $M\in\ulc$ to be the right derived functor $\textsf{R}^q(\Gamma)(M)$. This cohomology is denoted by $H^q_{\ulc}(M)$.
\end{defn}

 \subsection{Realization  Theorem of \texorpdfstring{$H^q_{\ulc}(M)$}{}   }\label{sec: extended CE complex} In the same way as Proposition \ref{lem: basic coho} we can define an extended Chevalley-Eilenberg complex and apply it to realize the cohomology of $\ulc$. We define
	$$C_{\ulc}^q(M):=\Hom_\calr(\calr\otimes_\bbf\bigwedge\nolimits^q\ggg, M)$$
for any $M\in\ulc$ and $q\in\mathbb{N}$. Then by a direct check it is easily seen that $C_\ulc^q(M)$ is still an object in $(\ggg,\calr)\hmod$ with unity $\calr$-action as in (\ref{eq: unity}).  The differential $\sfd_q^M: C_\ulc^q(M){\longrightarrow} C_\ulc^{q+1}(M)$ can be similarly defined as in Proposition \ref{lem: basic coho}(3) such that  $\bfC(M):=(C_{\ulc}^q(M), \sfd:=\sfd_q^M)$ becomes a cochain complex in  $(\ggg,\calr)\hmod$.

\subsubsection{Projective resolution} In order to establish the realization  theorem, we have to make some preparation. Set { $\bfV_q=\calr\natural U(\ggg)\otimes_\calr(\calr\otimes\bigwedge^q(\ggg))$ for $q\in\bbz_{\geq 0}$. Note that we have natural $\calr$-action $\theta$ and $\ggg$-action $\rho$} on $ \bfV_q$ with {
\begin{align}\label{eq: action on Vq}
 &\theta(f)(u\otimes (r\otimes w))=fu\otimes (r\otimes w),\cr
&\rho(E)(u\otimes (r\otimes w))=E u\otimes (r\otimes w)
\end{align}}
for {$f\in\calr, E\in \ggg, u\otimes (r\otimes w)\in \bfV_{q}$} with $u\in \calr\natural U(\ggg)$, $r\in\calr$ and $w\in \bigwedge^{q}\ggg$.

\begin{conven}\label{conv: z and y} In the subsequent arguments, we make an appointment on notations. Set
$$z=u\otimes r\otimes w\in \bfV_{q+1}$$
and
$$y=r\otimes w\in \calr\otimes\bigwedge\nolimits^{q+1}\ggg,$$
where $u\in \calr\natural U(\ggg)$, $r\in\calr$ and $w=g_1\wedge\cdots\wedge g_{q+1}\in\bigwedge^{q+1}\ggg$.
\end{conven}
Note that those elements $z$ in the above linearly span $\bfV_{q+1}$, and the ones such as $y$ linearly span $\calr\otimes \bigwedge^{q+1}\ggg$. We will use those $z$ and $y$ in the subsequent arguments.

\begin{lem}\label{lem: Vq is free LC} $\bfV_q$ is a free universal Lie-Cartan module.
\end{lem}

\begin{proof}
We only need to check the compatibility between $\rho$ and $\theta$.
For $z\in \bfV_q$, we want to show
\begin{align}\label{eq: bfv check}
[\rho(X),\theta(f)]z=\theta(X(f))z.
\end{align}
for $X\in\ggg$ and $f\in\calr$.
This directly follows from the definitions of $\rho$ and $\theta$.
 So $\bfV_q$ is indeed a universal Lie-Cartan module.
 Due to  Lemma \ref{lem: tensor r ulc} it is naturally a free universal Lie-Cartan module.
 \end{proof}

Note that { $\calr\natural U(\ggg)=\calr U(\ggg)\ggg\oplus \calr$}. Consider the canonical projection $$\kappa: \bfV_0=\calr\natural U(\ggg)\longrightarrow\calr.$$
Precisely,  $\kappa$ is defined via setting
 $$\kappa(u):=\gamma$$
for $u\in \calr\natural U(\ggg)$ uniquely expressed in the following way
\begin{align}\label{eq: u exp}
u=\sum_{i}\gamma_iu_i+\gamma
\end{align}
with $\gamma_i, \gamma\in \calr$ and $u_i\in U(\ggg)\ggg$.
It is easily checked that $\kappa$ is a Lie-Cartan module homomorphism.

Consider $d_0:\bfV_1\rightarrow \bfV_0$ defined via
$$d_0(u\otimes (r\otimes   g))=urg$$
and $d_q: \bfV_{q+1}\rightarrow \bfV_{q}$ defined for $q>0$ via
\begin{align*}
&d_q(u\otimes (r\otimes(  g_1\wedge\dots\wedge g_{q+1})))\cr
=&\displaystyle{\sum_{1\leq s<t\leq q+1}}(-1)^{s+t}u\otimes (r\otimes ([g_s,g_t]\wedge g_1\wedge\dots\wedge\hat g_s \wedge\dots \wedge\hat g_t \wedge \dots \wedge g_{q+1})) \cr
&+\sum_{1\leq s\leq q+1}(-1)^{s+1}urg_s\otimes (1\otimes (g_1\wedge\dots\wedge\hat g_s\wedge\dots\wedge g_{q+1}))			\end{align*}
for $g_1,\ldots,g_{q+1}\in\ggg$.

\begin{lem}\label{lem: dq is lc h} All $d_q$ are Lie-Cartan module homomorphisms for $q=0,1,\ldots$.
\end{lem}

\begin{proof} It is easily verified that $d_q$ is an $\calr$-module homomorphism. Now we check it is a $\ggg$-module homomorphism. It is obvious in the case $q=0$. In the following, we suppose $q>0$.
For $X\in\ggg$ and $z=u\otimes (r\otimes (g_1\wedge \cdots \wedge g_{q+1}))\in \bfV_{q+1}$, we have
\begin{align*}
&\rho(X) d_q (z)\cr
=&\displaystyle{\sum_{1\leq s<t\leq q}}(-1)^{s+t}X u\otimes (r\otimes ([g_s,g_t]\wedge g_1\wedge\dots\wedge\hat g_s \wedge\dots \wedge\hat g_t \wedge \dots \wedge g_{q+1})) \cr
&+\sum_{1\leq s\leq q}(-1)^{s+1}X urg_s\otimes (1\otimes (g_1\wedge\dots\wedge\hat g_s\wedge\dots\wedge g_{q+1}))\cr
=&d_q(\rho(X) z).
\end{align*}
The proof is completed.
\end{proof}

By Lemmas \ref{lem: Vq is free LC} and \ref{lem: dq is lc h}, we are in the position to introduce a projective resolution.

\begin{prop}\label{prop: proj resol}
 The following sequence gives rise to  a projective resolution of $\calr$ in $\ulc$:
\begin{align}\label{eq: proj res}
 \cdots\cdots\longrightarrow\bfV_{q+1}{\overset{d_q} \longrightarrow}\bfV_{q}\longrightarrow\cdots\longrightarrow\bfV_1{\overset{d_0}{\longrightarrow}}\bfV_0{\overset{\kappa}{\longrightarrow}}\calr
\longrightarrow 0.
\end{align}
\end{prop}
\begin{proof} Owing to Lemmas  \ref{lem: Vq is free LC} and \ref{lem: dq is lc h}, we only need to show the exactness of the sequence (\ref{eq: proj res}).  The arguments proceed in three steps.

(1) The exactness at $\calr$.

By definition,  it is easily seen that $\kappa$ is surjective. So it is proved.

(2) The exactness at $\bfV_0$, i.e.  $\ker\kappa=\im d_0$.

By definition it is easily shown
{
$$\ker\kappa=\calr U(\ggg)\ggg  =\im d_0.$$
}
 So it is proved.

(3) The exactness at $\bfV_q$ for $q>0$.

Consider
 $z=u\otimes (r\otimes w)\in \bfV_{q+1}$.
 Then $\bfV_{q+1}$ is spanned by monomial tensor elements of the form $z$. We express  $ur$ as $\gamma+\sum_{i}\gamma_iu_i$ with $\gamma, \gamma_i\in \calr$ and $u_i\in U(\ggg)\ggg$  by the arguments around (\ref{eq: u exp}).
 Then we have $$d_{q}(z)=d_{q}(\gamma\otimes (1\otimes w)+\sum_i\gamma_i u_i\otimes 1\otimes w).$$ Note that all $d_q$ are Lie-Cartan module homomorphisms. So we have
 \begin{align*}
 d_{q}(z)&=\gamma d_{q}(1\otimes1\otimes w) +\sum_i \gamma_i d_{q}(u_i\otimes (1\otimes w)).
 \end{align*}
Set $d'_{q}:=d_{q}|_{U(\ggg)\otimes \bigwedge^{q+1}\ggg}$. Then $\{d'_q\}$ satisfies the exactness because $d'_q$ is in the ordinary projective resolution of $\bbf$ in the $U(\ggg)$-module category (see \cite[\S7]{Wei}):
\begin{align}\label{eq: ordinary exact}
(\bfV^{{(\ggg)}}_q:=U(\ggg)\otimes \bigwedge\nolimits^q\ggg, d_q')_{q=0,1,2,\ldots}{\overset {\kappa|_{U(\ggg)}}\longrightarrow}\bbf\longrightarrow 0.
\end{align}
 Hence we have
\begin{align*}
 d_{q-1}\circ d_{q}(z)\
 =\gamma d'_{q-1}\circ d'_{q}(1\otimes1\otimes w) +\sum_i  \gamma_i d'_{q-1}\circ d'_{q}(u_i\otimes (1\otimes w))
 =0
 \end{align*}
which means $\im d_q\subseteq\ker d_{q-1}$.

Next we verify that $\im d_q\supseteq \ker d_{q-1}$. Take  it into account that $\bfV_q$ can be { simply expressed} as $\calr U(\ggg)\otimes_\bbf \bigwedge^{q}\ggg$, and all $d_q$ are Lie-Cartan module homomorphisms again.  Thus for any $\zzz\in \bfV_{q}$ we can write $\zzz=\sum_i\gamma_i u_i\otimes w_i$ with $\gamma_i\in \calr$, $u_i\in U(\ggg)$ and $w_i\in \bigwedge^{q}\ggg$. Then we are actually considering such $d_q$ with the following equivalent definition
$$d_{q-1}(\zzz)=\sum_i \gamma_i d_{q-1}'(\zzz_i)$$
for $\zzz_i=u_i\otimes w_i \in U(\ggg)\otimes \bigwedge^{q}\ggg$. Note that $\bfV_q$ now becomes free over $\calr$.
 Suppose $\zzz\in \ker d_{q-1}$. Then the assumption  $d_{q-1}(\zzz)=0$ is equivalent to the one that all $d_{q-1}'(\zzz_i)=0$ if $\gamma_i\ne 0$. By the exactness of $\{d_q'\}$ in (\ref{eq: ordinary exact}), we have all $\zzz_i\in \im d'_{q}$ whenever $\gamma_i\ne 0$. Keeping in mind again that $d_{q}$ is a Lie-Cartan module homomorphism, we finally have  $\zzz=\sum_i\gamma_i \zzz_i\in \im d_{q}$. Correspondingly, $\im d_q\supseteq \ker d_{q-1}$, and then $\im d_q=\ker d_{q-1}$. The exactness at $\bfV_q$ is proved.

Summing up, we accomplish the proof.
\end{proof}

\subsubsection{}
 Let $M\in \ulc$ and set $\bfH_q:=\Hom_{\calr\natural U(\ggg)}(\bfV_q,M)$. We consider  an $\calr$-module structure on $\bfH_q$.

\begin{lem} There is a natural $\calr$-module structure on $\bfH_q$ with unity $\calr$-action $\Omega$ defined via
$$\Omega(f)\psi(z)=\omega(f)\psi(z)$$
for any $\psi\in \bfH_q, f\in \calr, z\in \bfV_q.$
\end{lem}

\begin{proof} 
Let
$\theta'$ and $\rho'$  denote the $\calr$-module action  and $\calr\natural\ggg$-module action on $M$, respectively. In order to prove the lemma, we only need to check that $\Omega(f)\psi$ falls in $\bfH_q$ for $f\in \calr$ and $\psi\in\bfH_q$. It is equivalent to verify that (i) $\Omega(f)\psi$ is an $\calr$-module homomorphism; (ii) $\Omega(f)\psi$ is a $\calr\natural\ggg$-module homomorphism. For (i), we have
  \begin{align*}
(\Omega(f)\psi)(rr_1)&=\omega(f)\psi(rr_1)\cr
&=\omega(f)r\psi(r_1)\cr
&=r(\Omega(f)\psi)(r_1).
 \end{align*}
For (ii), we need to show that $\Upsilon(E)(\Omega(f)\psi)=0$ for $E\in \calr\natural\ggg$. Here $\Upsilon(E)$ denotes the Lie-action of $\calr\natural\ggg$  on $\bfH_q$ as usually defined via $\Upsilon(E)\varphi(z):=\rho'(E)(\varphi(z))-\varphi(\rho(E)z)$ where $\varphi\in \Hom_\bbf(\bfV_q,M)$,  $z=u\otimes r\otimes w\in \bfV_q$,  $\rho$ is defined as in (\ref{eq: action on Vq}). In fact,
by calculation,  we have
\begin{align*}
&\Upsilon(E)(\Omega(f)\psi)(z)\cr
=&\rho'(E)(\omega(f)\psi(z))-(\Omega(f)\psi)(\rho(E)z)\cr
=&\rho'(E)\omega(f)\psi(z)-\omega(f)\psi(\rho(E)z)\cr
=&\omega(f)(\Upsilon(E)\psi)(z)\cr
=&0.
\end{align*}
Hence $\Omega(f)\psi\in \bfH_q$. Note that $\omega$ is an algebra homomorphism. Hence $\Omega$ is a $\calr$-module action on $\bfH_q$.
\end{proof}


\subsubsection{Realization theorem of the  \texorpdfstring{$\ulc$}{}-cohomology}

Keeping the notations as above,  we are in a position to introduce one of the main results in this section.

\begin{thm}\label{thm: cohom thm} Let $M\in\ulc$.
 Then the cohomology module $H^q_{\ulc}(M)$  are the cohomology of the cochain complex $\bfC(M)$. This  means, as $\calr$-modules
$$H^q_\ulc(M)=H^q(\bfC(M)).$$
\end{thm}

\begin{proof} In view of Proposition \ref{prop: proj resol}, we consider the projective resolution of $\calr$ in $\ulc$:
 $$\bfV_\bullet:= (\bfV_q, d_q)_{q=0,1,2,\ldots}{\overset\kappa\longrightarrow}\calr\longrightarrow 0.$$
 By definition, $H^q_{\ulc}(M)=\textsf{R}^q(\Gamma)(M)$ which is equal to  $\mathsf{R}^q(\Gamma'(\bfV_\bullet))$ for $\Gamma':=\Hom_{\ulc}(-,M)$. Define a map $\Psi$:
 \begin{align*}
 \Gamma'(\bfV_q)=\Hom_{\calr\natural U(\ggg)}
 (\calr\natural U(\ggg)\otimes_\calr\calr\otimes_\bbf\bigwedge\nolimits^q\ggg,M)
 &{\longrightarrow} \Hom_{\calr}(\calr\otimes_\bbf \bigwedge\nolimits^q\ggg, M)=C_{\ulc}^q(M)\cr
 \psi&\mapsto \Psi(\psi)
 \end{align*}
   via letting
 $$\Psi(\psi)(y)=\psi(1\otimes r\otimes w)$$
 for any $\psi\in\bfH_q$, and  $y=r\otimes w\in \calr\otimes_\bbf\bigwedge^q\ggg$.
 By a straightforward calculation, it is readily verified that $\Psi$ is an $\calr$-module isomorphism.

 Up to the $\calr$-isomorphism $\Psi$, we can identify both sides of the above isomorphism map.
 What remains is to prove $\Gamma'(d_q)=\sfd_q^M$. For this, take any $\phi\in C_\ulc^{q}(M)=\Hom_\calr(\calr\otimes_\bbf \bigwedge^{q}\ggg,M)$ and confirm $\Gamma'(d_q)(\phi)=\sfd_q^M(\phi)$. For any $y=r\otimes w\in \calr\otimes\bigwedge^{q+1}\ggg$ with $r\in\calr, w=g_1\wedge\cdots\wedge g_{q+1}\in\bigwedge^{q+1}\ggg$,
 we have  $\Gamma'(d_q)(\phi): y\mapsto  \Psi^{-1}(\phi)\circ d_q(1\otimes y)$. Precisely,
 \begin{align*}
&\Psi^{-1}(\phi)\circ d_q(1\otimes y)\cr
=&
\sum_{1\leq s<t\leq q+1}(-1)^{s+t}\Psi^{-1}(\phi)(1\otimes (r\otimes ([g_s,g_t]\wedge g_1\wedge\dots\wedge\hat g_s \wedge\dots \wedge\hat g_t \wedge \dots \wedge g_{q+1}))) \cr
&+\sum_{1\leq s\leq q+1}(-1)^{s+1}\Psi^{-1}(\phi)(rg_s\otimes (1\otimes (g_1\wedge\dots\wedge\hat g_s\wedge\dots\wedge g_{q+1})))\cr
=&
\sum_{1\leq s<t\leq q+1}(-1)^{s+t}\phi(r\otimes ([g_s,g_t]\wedge g_1\wedge\dots\wedge\hat g_s \wedge\dots \wedge\hat g_t \wedge \dots \wedge g_{q+1})) \cr
&+\sum_{1\leq s\leq q+1}(-1)^{s+1}rg_s\phi(1\otimes (g_1\wedge\dots\wedge\hat g_s\wedge\dots\wedge g_{q+1}))\cr
=& \sfd_q^M(\phi)(y).
  \end{align*}
Hence, $\Gamma'(d_q)=\sfd_q^M$.

We finally obtain the equality  $H^q_\ulc(M)=H^q(\bfC(M))$.	
\end{proof}

 \subsection{Extension ring of \texorpdfstring{$\calr$}{}  in the \texorpdfstring{$\ulc$}-cohomology}
 		
 \subsubsection{Extensions in the \texorpdfstring{$\ulc$-}{} cohomology}
For \texorpdfstring{$L,M\in\ulc$}{} , we may consider \texorpdfstring{$\Ext^q_\ulc(L,M)$.}{}   Let $L$ admit a projective resolution $\bfV_\bullet^L=(V^L_q, d_q)$, $q=0,1,\ldots$ in $\ulc$, this means the following long exact sequence
$$\bfV_\bullet^L {\overset{\kappa}\longrightarrow}L\longrightarrow 0$$
with  $V^L_q$ being projective objects in $\ulc$, $\kappa$ and all $d_q$ being homomorphisms in $\ulc$.  As in the usual way,
$\Ext^q_\ulc(L,M)=\textsf{R}^q(\Gamma'(L))$ for $\Gamma'=\Hom_\ulc(-,M)$ which is a left exact (contravariant) functor from $\ulc$ to $(\ggg,\calr)\hmod$. Hence $\Ext^q_\ulc(L,M)=H^q(\Gamma'(\bfV_\bullet^L))$.

In particular, take $L=\calr$.  We can determine  $\Ext^q_\ulc(\calr, M)$ as follows
\begin{align}\label{eq: ext def}
\Ext^q_\ulc(\calr,M)=H^q_\ulc(\bfC(M)).
\end{align}

\subsubsection{}						
\begin{prop}\label{prop: 5.8}  We have the following isomorphism of $\calr$-modules
$$\Ext_\ulc^\bullet(\calr,\calr)\cong H^\bullet(\ggg, \calr).$$
\end{prop}
\begin{proof} By Theorem \ref{thm: cohom thm} and (\ref{eq: ext def}), we have $$\Ext_\ulc^q(\calr,\calr)=H^q_\ulc(\bfC)=H^q_\ulc(\calr).$$
 By Theorem \ref{thm: cohom thm} and Proposition \ref{lem: basic coho}(4), we have $\Ext_\ulc^q(\calr,\calr)\cong H^q(\ggg,\calr)$. The desired statement is consequently proved.
\end{proof}						
		
}					

{
\subsection{A \texorpdfstring{$\ulc$}{} -cohomology theorem for \texorpdfstring{$\ggg=W(n)$}{} }

\subsubsection{} In \cite{GF} (or see  \cite[Theorems 2.2.7 and 2.2.8]{Fuks} for details), I. M. Gelfand and D. B. Fuks first proved a cohomology theorem  in the case when $\calr$ is taken place by the corresponding  power series algebra.  Here we make a brief account that the following result, as a corollary to Gelfand-Fuks' theorem, still holds in the case of $\calr$.

\begin{prop}\label{thm: G-F} For $\ggg=W(n)$, $H^\bullet(\ggg,\calr)\cong H^\bullet(\mathfrak{gl}(n))$ for the ordinary Lie algebra cohomology.
\end{prop}
\begin{proof} Recall that $\nabla(\lambda)=\Hom_{U(\ggg_{\geq0})}(U(
		\ggg), L^0(\lambda))$ for $\lambda\in\Lambda^+$. By a classical result of Lie algebra cohomology (see for example \cite[Theorem 1.5.4]{Fuks}, \cite[Theorem 6.9]{Kn}), we have $H^q(\ggg,\nabla(\lambda))=H^q(\ggg_{\geq0}, L^0(\lambda))$.

 Note that $\mathfrak{gl}(n)\cong \ggg_0\hookrightarrow \ggg_{\geq0}$. Set $\ggg^+:=\ggg_{\geq0}$ and $\ggg^{\sharp}:=\ggg_{>0}$. Then $\ggg^\sharp$  is an ideal of $\ggg^+$ with $\ggg_0=\ggg^+\slash \ggg^\sharp$.
By Serre-Hochschild spectral sequence theorem (see \cite[Theorem 1.5.1]{Fuks} or \cite[\S7.5]{Wei} for details), there exists a spectral sequence
$$\{E^{p,q}_r, d_r^{p,q}: E_r^{p,q}\longrightarrow E_r^{p+r,q-r+1}\}$$
with $E_1^{p,q}=H^q(\ggg^+, \Hom(\bigwedge^p(\ggg^+\slash\ggg^\sharp),L^0(\lambda))$, and
 \begin{align*}
 E_2^{p,q}=&H^p(\ggg^+\slash \ggg^\sharp, H^q(\ggg^\sharp, L^0(\lambda))
 \Longrightarrow H^{p+q}(\ggg^+, L^{0}(\lambda)).
\end{align*}
Note that $L^0(\lambda)$ is regarded a trivial $\ggg^\sharp$-module, and $\ggg^\sharp$ itself is a $\ggg_0$-module.  We further have
 \begin{align*}
 E_2^{p,q}=&H^p(\mathfrak{gl}(n), H^q(\ggg^\sharp, L^0(\lambda))
  \cr
 =&H^p(\mathfrak{gl}(n))\otimes (H^q(\ggg^\sharp)\otimes L^0(\lambda))^{\gl(n)}.
\end{align*}
 When $\lambda$ is zero weight, equivalently to say,  $L^0(\lambda)=\bbf$ it can be shown that   $E_\infty=E_2$ in the same way as in the situation when $\calr$ is taken place of the corresponding divided power algebra (see the proof of Theorem 2.2.8 of \cite{Fuks}). Furthermore,  $E_1^{p,q}=0$ when $p\ne 0$.
Hence $E^{p,q}_\infty=E_1^{p,q}$ and
\begin{align*}
H^m(\ggg^+, \bbf)=E_1^{0,m}=
H^m(\mathfrak{gl}(n), \bbf).
\end{align*}
So we have $H^\bullet(\ggg,\calr)=H^\bullet(\ggg^+,\calr)  =  H^\bullet(\mathfrak{gl}(n))$ when $\lambda=0$.
\end{proof}
}
\subsubsection{}
\begin{thm}\label{thm: third main thm} Let $\ggg=W(n)$. Then the extension ring $$\Ext_\ulc^\bullet(\calr,\calr)\cong H^\bullet(\mathfrak{gl}(n,\bbf)).$$
\end{thm}

\begin{proof}
This theorem follows from Proposition \ref{prop: 5.8} and Proposition \ref{thm: G-F}.
\end{proof}

\begin{rem} In comparison of cohomology computations, the extension $\Ext_\ulc^q(\calr,\calr)$ in the $\ulc$-cohomology is completely different from the one $\Ext^q_{U(\ggg)}(\calr,\calr)$ in the ordinary cohomology of Lie algebras.
\end{rem}

\subsubsection{} We turn back to Question \ref{ques: ext}(1) with respect to $\ulc$ for $\ggg=W(n)$,  giving the following  answer. 		

	\begin{cor}\label{cor: ext ring results} The extension ring $\uE:=\Ext_\ulc^\bullet(\calr,\calr)$ is finite dimensional.
 \end{cor}	
 \begin{proof} It is a direct consequence of Theorem \ref{thm: third main thm} and the finite dimensional property of $H^\bullet(\mathfrak{gl}(n,\bbf))$ (see for example, \cite[Theorem 2.1.1]{Fuks}).
 \end{proof}

\subsection*{Acknowledgements}
This work is partially supported by the National Natural Science Foundation of China (Grant Nos. 12071136 and 12271345), the Science and Technology Commission of Shanghai Municipality (No. 22DZ2229014), the Science and Technology Commission of Shanghai Municipality-Shanghai Local University Capacity Building Project (No. 23010502100), the Hebei Natural Science Foundation of China (No. A2021205034) and the Special Project on Science and Technology Research and Development Platforms, Hebei Province (No. 22567610H).

\end{document}